\documentclass{article}
\usepackage{amsmath}
\usepackage{amssymb}
\usepackage{enumerate}
\usepackage{amsthm}
\usepackage{todonotes}
\usepackage{mathtools} 
\usepackage[UKenglish]{babel}
\usepackage[T1]{fontenc}
\usepackage[utf8]{inputenc}
\usepackage{lmodern} 
\usepackage{tikz}
\usepackage{thmtools} 
\usepackage{thm-restate} 
\usepackage{stackrel} 
\usepackage{hyperref}
\usepackage[left = 23mm,right=23mm, top = 23mm,bottom =23mm, paper = a4paper]{geometry}
\usetikzlibrary {arrows.meta}
\usetikzlibrary{decorations.pathreplacing}
\usepgflibrary {shadings}

\declaretheorem[name=Theorem,numberwithin=section]{thm} 
\newtheorem*{thm*}{Theorem}

\newtheorem*{define*}{Definition}
\newtheorem{define}[thm]{Definition}

\newtheorem*{lemma*}{Lemma}
\newtheorem{lemma}[define]{Lemma}

\newtheorem*{algorithm*}{Algorithm}

\newtheorem*{construction*}{Construction}

\newtheorem*{prop*}{Proposition}
\newtheorem{prop}[define]{Proposition}

\newtheorem*{obs*}{Observation}

\newtheorem*{fact*}{Fact}

\newtheorem*{remark*}{Remark}
\newtheorem{remark}[define]{Remark}

\newtheorem*{claim*}{Claim}
\newtheorem{claim}[define]{Claim}

\newtheorem*{quest*}{Question}

\newtheorem*{cor*}{Corollary}

\newtheorem*{conjecture*}{Conjecture}

\newtheorem*{question*}{Question}
\newtheorem{question}[define]{Question}

\newtheorem*{example*}{Example}

\usepackage[mathlines]{lineno}
\usepackage[normalem]{ulem}
\usepackage{graphicx}

\usepackage{color}
\definecolor{grey}{rgb}{.7,.7,.7}
\definecolor{blue}{rgb}{0,0,.8}
\definecolor{red}{rgb}{.8,0,0}
\definecolor{green}{rgb}{0,.4,0}
\definecolor{gold}{rgb}{0.8,0.6,0.1}
\definecolor{brown}{rgb}{0.8,0.4,0.1}
\definecolor{arxivcolor}{rgb}{0.5,0.5,0}
\definecolor{journalcolor}{rgb}{0.5,0,1}
\definecolor{purple}{rgb}{0.6,0.2,0.6}
\definecolor{pastelgreen}{rgb}{0,0.65,0.1}

\newcommand{\R}{\mathbb{R}}

\newcommand{\Z}{\mathbb{Z}}
\newcommand{\N}{\mathbb{N}}

\newcommand{\leb}{\mc{L}}
\newcommand{\sep}{\operatorname{sep}}

\newcommand{\floor}[1]{\lfloor#1\rfloor}

\DeclareMathOperator{\diam}{diam}

\DeclareMathOperator{\dist}{dist}
\newcommand{\abs}[1]{\left|#1\right|}

\newcommand{\norm}[1]{\left\|#1\right\|}

\newcommand{\mc}[1]{\mathcal{#1}}

\newcommand{\set}[1]{\left\{#1\right\}}
\newcommand{\br}[1]{\left(#1\right)}
\newcommand{\sqbr}[1]{\left[#1\right]}

\newcommand{\upp}[1]{^{(#1)}}

\title{Fast repetitivity in non-rectifiable Delone sets}
\author{Ashwin Bhat 
   and Michael Dymond}  
   \date{}
   
\begin{document}
\maketitle

\begin{abstract}
    We present a construction of non-rectifiable, repetitive Delone sets in every Euclidean space $\R^{d}$ with $d\geq 2$. We further obtain a close to optimal repetitivity function for such sets. The proof is based on the process of `encoding' a non-realisable density in a Delone set, due to Burago and Kleiner.
\end{abstract}
\section{Introduction}
Shechtman's discovery that quasicrystals exist~\cite{schechtman1} won him the Nobel Prize for Chemistry in $2011.$ Over the four decades surrounding this remarkable event, studying such structures became of huge interest to mathematicians through the language of Delone sets (\cite{adiceam2016open}, \cite{baake2002guide}). Quasicrystals are materials whose x-ray diffraction spectra contain defined spots which demonstrate long-range order, but unlike the classical theory of crystalline chemistry, its patterns are aperiodically ordered: typically, the symmetries they have are impossible for fully periodic atomic structures to determine \cite{lagarias2003repetitive}. Analogously, Delone sets that model the atomic structure of quasicrystals have the formal property of being \emph{repetitive} (see Definition \ref{def: rep_function}). This determines a \emph{repetitivity function} which roughly measures how far one must look from any point in space to find every pattern of a given size $r$ in the Delone set. This function $R(r)$ describes, in the sense of functional growth rate, how `fast' the repetitivity of a Delone set is. 

Repetitive Delone sets are further a significant object of interest in dynamical systems, since they give rise to minimal topological dynamical systems, see e.g.~\cite[Theorem~3.11]{frettloh} or \cite[Theorem~6.5]{SS_substition}. Moreover, the repetitivity function, as well as the related patch counting function, of repetitive Delone sets have been well-studied and shown to relate to important concepts from mathematical physics and dynamical systems. For example, \emph{ideal crystals} (Delone sets with a full rank of translation symmetries) are characterised by boundedness of the repetitivity function~\cite{lagarias2003repetitive}. Linearly repetitive Delone sets (with $R(r)\in O(r)$) are also well-distinguished in the literature, such as \cite{lagarias2003repetitive}, because they typically arise in constructions of quasicrystals and they include all Delone sets corresponding to self-similar tilings~\cite{priebe2001characterization}.

Delone sets have also been studied from the point of view of discrete metric spaces, where one of the key objectives is, broadly speaking, to compare different discrete metric structures. Two metric spaces $M$ and $N$ are said to be \emph{bi-Lipschitz equivalent} if there exists a bi-Lipschitz bijection $f\colon M\to N$. Since Lipschitz mappings are those which do not increase distances by more than some finite factor, bi-Lipschitz equivalent spaces may be thought of as distorted copies of each other with at most a constant distortion factor. A major line of research on Delone sets began with the simple but highly non-trivial question of whether every pair of Delone sets in $\R^{d}$ are bi-Lipschitz equivalent, or put differently, whether every Delone set in $\R^{d}$ is \emph{rectifiable}, or bi-Lipschitz equivalent to the integer lattice $\Z^{d}$. This question was posed by Furstenberg in the context of dynamical systems (see \cite{burago2002rectifying} for a detailed discussion), and later by Gromov \cite{gromov1992asymptotic} from the standpoint of metric geometry. 

The question was finally resolved in the negative in 1998 independently by Burago and Kleiner~\cite{burago1998separated} and McMullen~\cite{mcmullen1998lipschitz}: in every Euclidean space with dimension $d \geq 2$ there are non-rectifiable Delone sets. A remarkable aspect of both the proofs of \cite{burago1998separated} and \cite{mcmullen1998lipschitz} is that the problem is transformed from `discrete' to `continuous'.
More precisely, the existence of non-rectifiable Delone sets is shown via the construction of so-called (bi-Lipschitz) \emph{non-realisable densities}, that is, measurable functions $\rho:[0,1]^d \to \R_{>0}$ (or $\rho\colon \R^{d}\to \R_{>0}$ in the case of \cite{mcmullen1998lipschitz}) having no bi-Lipschitz solution $f:[0,1]^d \to \R^d$ (or $f\colon\R^{d}\to\R^{d})$ to the `prescribed Jacobian equation'
\begin{equation} \label{eq: prescribed volume form}
    \rho = \text{Jac}(f) \qquad \text{almost everywhere}.
\end{equation}
The solvability of the prescribed Jacobian equation (and its generalisations) was already, and continues to be, a topic of wide independent interest; see for example \cite{dacorogna2007direct}, \cite{riviere1996resolutions}, \cite{dymond2018mapping} and \cite{dymond2023highly}. The important discovery of bi-Lipschitz non-realisable densities, suitable for establishing the existence of non-rectifiable Delone sets is due to Burago and Kleiner~\cite{burago1998separated} and McMullen~\cite{mcmullen1998lipschitz}.

Having obtained a non-realisable density $\rho\colon [0,1]^{d}\to \R_{>0}$, Burago and Kleiner~\cite{burago1998separated} construct a Delone set $X$ by placing `better and better' discrete approximations of $\rho$ alongside each other. In the present work, we give this approximation property a name--we say that $X$ \emph{encodes} the density $\rho$:
\begin{restatable}{define}{encodingdefinition} \label{defn: X encoding rho}[inspired by \cite{burago1998separated}]
    We say that a Delone set $X \subset \R^d$ \emph{encodes} a measurable density $\rho:[0,1]^d\to \R_{>0}$ if there exists a sequence $r_n\to \infty,$ a sequence of closed cubes $Q_n \subset \R^d$ with sidelength $r_n$ and a sequence of bijective affine maps $\varphi_n:[0,1]^d \to Q_n$ such that if 
    \begin{equation*} \label{eqn: X encoding rho}
        \mu_n(A) \coloneqq  \frac{|X \cap \varphi_n(A)|}{r_n^d} \qquad \emph{for each} \;A \subseteq [0,1]^d,
    \end{equation*} then $\mu_n \rightharpoonup \rho \mathcal{L}$ on $[0,1]^d$, where $\leb$ denotes the Lebesgue measure and $\rho\leb$ stands for the measure given by the formula $\rho\leb(A)=\int_{A}\rho\,d\leb$.
\end{restatable}
The notion of a Delone set encoding a density $\rho$ is important due to the following result of Burago and Kleiner:
\begin{restatable}{thm}{realisable} \label{prop:non-rectifiable non-realisable}[Burago and Kleiner~\cite[Lemma~2.1]{burago1998separated}, see also \cite[Lemma~3.4]{dymond2023highly}]
        \newline Let $X \subset \R^d$ be a Delone set and $\rho:[0,1]^d \to \R_{> 0}$ be a measurable function which is non-realisable. If $X$ encodes $\rho,$ then $X$ is non-rectifiable. 
    \end{restatable}
    We do not know whether the converse to Theorem~\ref{prop:non-rectifiable non-realisable} is true: 
    \begin{question} \label{question}
        If a Delone set $X\subset \R^{d}$ is non-rectifiable, does $X$ encode a non-realisable density $\rho$? 
    \end{question}
Motivated by the major research interest in repetitive Delone sets and by Theorem~\ref{prop:non-rectifiable non-realisable}, we will study repetitivity in Delone sets encoding a given density $\rho$. Indeed, the interaction of rectifiability and repetitivity has already attracted significant investigation - see \cite{Smilansky_2022}, \cite{inoquio2024rectifiability}, \cite{cortez2014some} and \cite{aliste2013linearly}. In order for a Delone set $X$ to encode a given density $\rho$, $X$ must contain patterns of points of a certain form, corresponding to the measures $\mu_n$ in Definition~\ref{defn: X encoding rho}. On the other hand, repetitivity places a strong restriction on the types of patterns that can occur within $X$. This leads to the question of whether the two properties can occur simultaneously in a Delone set $X$. Cortez and Navas~\cite{cortez2014some} answer this question negatively, constructing repetitive, non-rectifiable Delone sets in $\R^2$ (generalising to all $\R^{d}$ with $d\geq 2$). Cortez and Navas~\cite[Remark~14]{cortez2014some} further raise the hard question of determining the class of repetitivity functions admitting non-rectifiable, repetitive Delone sets. The construction of \cite{cortez2014some} does not provide bounds on the repetitivity function, except along a subsequence. Otherwise, there have been important advances on this problem from the opposite direction: Aliste-Prieto, Coronel and Gambaudo \cite{aliste2013linearly} prove, using a criterion of Burago and Kleiner \cite[Theorem $1.3$]{burago2002rectifying}, that every linearly repetitive Delone set is rectifiable. A very recent work~\cite[Theorem B]{inoquio2024rectifiability} improves on this result: Inoquio-Renteria and Viera show that whenever $q \in [0,\frac{1}{d})$ that every $r(\log{r})^q-$repetitive Delone set is rectifiable. A different type of improvement is made by Smilansky and Solomon~\cite{Smilansky_2022}, who show that linear repetitivity may be weakened to `$\varepsilon$-linear repetitivity' in a sufficient condition for rectifiability.

In the present article we provide a short new construction of repetitive, non-rectifiable Delone sets, encoding a given density $\rho$, together with an explicit asymptotic bound on the repetitivity function.
\begin{restatable}{thm}{mainlemma}
        \label{delone}
   Let $d\in \N_{\geq 2}$ and $\rho:[0,1]^d \to [\frac{4}{3},\frac{5}{3}]$ be a measurable function. 
        Then there exists a repetitive Delone set in $\R^d$ which encodes $\rho$ and has repetitivity function
        \begin{equation*}
            R(r)\in\bigcap_{q>\frac{2}{d}}o\left(r \left(\frac{\log r}{\log \log r}\right)^q\right).
        \end{equation*}
        \end{restatable}
Theorem~\ref{delone}, together with Theorem~\ref{prop:non-rectifiable non-realisable} and the non-realisable densities given by Burago and Kleiner~\cite{burago1998separated}, delivers the first examples of repetitive, non-rectifiable Delone sets with explicit bounds on the repetitivity function. Since, according to \cite[Theorem~B]{inoquio2024rectifiability}, any repetitivity function $R$ of a non-rectifiable Delone set must satisfy $R(r)\in\displaystyle \bigcap_{q<\frac{1}{d}}\omega\left(r \left(\log r\right)^q\right)$, the asymptotic bound on $R(r)$ in our result is close to optimal.
\begin{restatable}{thm}{main} \label{1}
    Let $d \in \N_{\geq 2}$. Then there exists a non-rectifiable repetitive Delone set in $\R^d$ with repetitivity function
    \begin{equation*}
            R(r)\in\bigcap_{q>\frac{2}{d}}o\left(r \left(\frac{\log r}{\log \log r}\right)^q\right).
        \end{equation*}
\end{restatable}
The proof of Theorem~\ref{1}, using Theorems~\ref{prop:non-rectifiable non-realisable} and \ref{delone} and \cite[Theorem~1.2]{burago2002rectifying} is very quick. Therefore, we present it immediately here.
\begin{proof}[Proof of Theorem~\ref{1}]
Choose $\rho:[0,1]^d \to [\frac{4}{3},\frac{5}{3}]$ to be a non-realisable density, given by Burago and Kleiner in \cite[Theorem~1.2]{burago1998separated}. Strictly speaking, \cite[Theorem~1.2]{burago1998separated} provides a non-realisable density $\rho\colon [0,1]^{2}\to [1,1+c]$. It is a trivial matter to adjust the range of values to $[\frac{4}{3},\frac{5}{3}]$ and whilst the results and proofs of \cite{burago1998separated} are given in dimension $d=2$, Burago and Kleiner assure (in a remark directly following \cite[Theorem~1.2]{burago1998separated}) that they are valid, with only minor modifications to the proofs, in all higher dimensions. Note also that the existence of a non-realisable density $\rho\colon [0,1]^{d}\to [\frac{4}{3},\frac{5}{3}]$ is a consequence of \cite[Theorem~4.8]{dymond2018mapping}, where the proof is written in general dimension $d\geq 2$.
By Theorem~\ref{delone}, there exists a Delone set $X$ which is $o(r (\frac{\log r}{\log \log r})^q)-$repetitive for each $q>\frac{2}{d}$ which encodes $\rho$. By Theorem~\ref{prop:non-rectifiable non-realisable}, $X$ is non-rectifiable.
\end{proof}
In the final section of the paper, we explore the optimality of the repetitivity of the Delone set in Theorem \ref{delone}. Theorem \ref{thm: repetitivity implies constant rho} shows how a sufficiently repetitive Delone set can only encode densities which are constant almost everywhere. This shows the restriction faced in the `encoding method' to construct non-rectifiable Delone sets with fast repetitivity. 
\begin{restatable}{thm}{rlogr} \label{thm: repetitivity implies constant rho} Let $d\geq 2$, $\rho:[0,1]^d \to \R_{>0}$ a density and $X$ a repetitive Delone set in $\R^d$ with repetitivity function
\begin{equation*}
    R(r) \in O(r(\log r)^\frac{1}{d}).
\end{equation*} If $X$ encodes $\rho$, then $\rho$ is constant almost everywhere.
\end{restatable}
It would be very interesting if Question \ref{question} were true since Theorem \ref{thm: repetitivity implies constant rho} would rule out the existence of an $O(r(\log r)^{\frac{1}{d}})-$repetitive non-rectifiable Delone set, which would be an improvement on \cite[Theorem B]{inoquio2024rectifiability}.
\subsection{Notation and preliminaries}
\paragraph{Sets and measures} The dimension $d \in \N_{\geq 2}$ is treated as an unspecified constant throughout the paper.
We use $||\cdot||_2$ and $||\cdot||_\infty$ to denote the Euclidean norm and the supremum norm respectively. We use $\text{dist}$ to denote the Euclidean metric in $\R^d$. We write $B(x,r)$ to denote the open ball of radius $r>0$ centred at $x\in \R^d$ with respect to the $\text{dist}$ metric. We briefly write for sets $A,B \subset \R^d$ that their diameter, separation and distance respectively are
\begin{equation*}
       \diam(A) \coloneqq \sup_{x,y \in A}\dist(x,y),\qquad \sep(A) \coloneqq \inf_{x,y \in A}\dist(x,y) \qquad \text{and} \qquad
        \text{Dist}(A,B) \coloneqq \inf_{x \in A, y \in B}\dist(x,y).
\end{equation*}
For $n \in \mathbb{N},$ we write $[n]$ to denote the set $\{1,...,n\}.$
We say a property on a measure space holds \emph{almost everywhere} or \emph{a.e.} if it occurs everywhere on its domain except on a set of measure zero. 
Given a function $g:\R^d\to \R^m$ and measure $\mu$ on $\R^d,$ the pushforward measure $g \# \mu$ on $\R^m$ is defined as
\begin{equation*}
    g \# \mu(A) \coloneqq \mu(g^{-1}(A)), \qquad A \subset \R^m.
\end{equation*}
 We write $\mathcal{L}$ to denote the $d-$dimensional Lebesgue measure on $\mathbb{R}^d.$ For a measurable function $\rho:[0,1]^d \to \R_{>0},$ the measure $\rho\mathcal{L}$ is defined as
\begin{equation*}
    \rho \mathcal{L}(A) \coloneqq \int_{A}\rho\; d\mathcal{L}, \qquad A \subseteq [0,1]^d.
\end{equation*}
Given a measure $\mu$ and a sequence of measures $\mu_n$ on $U \subset \R^d$, we say $\mu_n$ weakly converges to $\mu,$ written $\mu_n \rightharpoonup \mu,$ if for every bounded continuous function $f:U \to \R$
\begin{equation*}
    \int_U f\;d\mu_n \xrightarrow{n\rightarrow\infty} \int_U f \; d\mu.
\end{equation*}
\paragraph{Delone sets}
\begin{define}[\cite{lagarias2003repetitive}, Definition $1.1$] A \emph{separated net}, also known as a \emph{Delone set}, is a discrete set $X \subset \R^d$ satisfying the following two criteria:
\begin{itemize}
    \item (Uniform discreteness) There exists an $r>0$ such that every ball of radius $r$ in $\R^d$ contains at most one point of $X.$
    \item (Relative density) There exists an $R>0$ such that every ball of radius $R$ in $\R^d$ contains at least one point of $X.$
\end{itemize}
    A Delone set $X \subset \R^d$ is called \emph{non-rectifiable} if it admits no bi-Lipschitz bijection with $\Z^d.$
\end{define}
\begin{define}[\cite{lagarias2003repetitive}, $(1.2)$, Definition $1.5$] \label{def: rep_function}
    Let $X$ be a Delone set, $x \in X$ and $r>0.$ An \emph{$r-$patch} centred at $x$ is denoted as $\mathcal{P}_{x,r}$ to be
    \begin{equation*}
        \mathcal{P}_{x,r} \coloneqq X \cap B(x,r).
    \end{equation*}
An \emph{$X-$translate} of $\mathcal{P}_{x,r}$ is an $r-$patch $\mathcal{P}_{y,r}$ with $y \in X$ and $$\mathcal{P}_{y,r} = y-x+\mathcal{P}_{x,r}.$$ Occasionally, when $r$ is specified, we abbreviate the notation $\mathcal{P}_{x,r}$ to $\mathcal{P}_{x}$.

    Given a non-decreasing function $R:(0,\infty) \to (0,\infty)$, we say $X$ is \emph{$R(r)-$repetitive} if for each $r>0$ and each pair of points $x\in X$ and $y\in\R^{d}$, an $X-$translate of the $r-$patch $\mathcal{P}_{x,r}$ intersects $B(y,R)$. The function $R:(0,\infty) \to (0,\infty)$ is called a \emph{repetitivity function} of $X$ and we say that $X$ is $R(r)$-repetitive.
\end{define}
\paragraph{Mappings}
Given two normed spaces $(X,||\cdot||_X)$ and $(Y,||\cdot||_Y)$ and $L>0,$ we say a function $f:X \to Y$ is \emph{$L$-Lipschitz} if for every pair $x,y \in X$ $$||f(x)-f(y)||_Y \leq L||x-y||_X.$$ The smallest such $L$ is called the \emph{Lipschitz constant} of $f$ and is denoted $\text{Lip}(f).$ Further, we say $f$ is \emph{$L-$bi-Lipschitz} if both $f$ and $f^{-1}$ are $L-$Lipschitz.

We refer to a measurable real-valued function $\rho:[0,1]^d \rightarrow \mathbb{R}_{>0}$ as a \textit{density}. Such a density is called \emph{non-realisable} if it admits no bi-Lipschitz solution $f:[0,1]^d \to \R^d$ to $$\rho = \text{Jac}(f)\qquad \text{almost everywhere}.$$

\begin{define}\label{def:rep_mapp}
	Let $M$ be a set. We call a mapping $\Psi\colon \Z^{d}\to M$ \emph{repetitive} if for every $r>0$ there exists $R(r)\geq r$ such that for every pair $x,y\in \Z^{d}$ there exists $w\in\Z^{d}$ such that
	\begin{equation*}
		B(w,r)\subseteq B(y,R) \qquad \emph{and} \qquad
		\Psi(w+z)=\Psi(x+z)
	\end{equation*}
	for all $z\in \Z^{d}\cap B(0,r)$. The function $R\colon (0,\infty) \to (0, \infty)$ is called a \emph{repetitivity function} of $\Psi$ and we say that $\Psi$ is $R(r)$-repetitive.
\end{define}
Finally, for non-negative real-valued functions $f$ and $g,$ we will write
    $f(x) \in o(g(x))$ (equivalently, $g(x) \in \omega(f(x))$) if $\lim_{x \to \infty} \frac{f(x)}{g(x)} = 0,$ and $f(x) \in O(g(x))$ if there exists a constant $C>0$ such that $f(x) \leq Cg(x).$

\section{A general construction of a repetitive Delone set}
There are two notions of repetitivity we use interchangeably, one for functions on the integer lattice and, more usually used in the literature, one for Delone sets. Structurally they depict similar concepts, that a patch in either sense can be found in any ball of a certain size relative to that patch. We formalise this in the following lemma showing that the repetitivity functions of a function and its analogous Delone set have the same growth rate.
\begin{lemma} \label{rep from mapping to net}
	Let $M$ be a non-empty finite set, $\Psi\colon \Z^{d}\to M$ be a repetitive mapping and for each $m\in M$ let $V_{m}$ be a non-empty subset of $[0,1]^{d}$ satisfying that
	\begin{equation}\label{eq:sep_cond}
		\min_{m\in M}\min\set {\sep(V_{m}),\emph{Dist} \br{V_{m},\partial[0,1]^{d}}}>0.
	\end{equation}
Then the set
	\begin{equation*}
X:=\bigcup_{z\in\Z^{d}}z+V_{\Psi(z)}
	\end{equation*}
	is a repetitive Delone set. Furthermore, if $\Psi$ is $R(r)-$repetitive, then $X$ is $(R(r+3\sqrt{d})+\sqrt{d})$-repetitive. 
\end{lemma}
\begin{proof}
    Notice that each unit volume cube with vertices in $\Z^d$ contains at least one point of $X$ since for each $m\in M$ the set $V_m$ is non-empty, so $X$ is relatively dense. Moreover, $X$ is uniformly discrete, since for any $x,y\in X$ with $x\neq y$ we have that $\text{dist}(x,y)$ is at least the quantity of \eqref{eq:sep_cond}. To show $X$ is repetitive, fix $r>0$, $x \in X$ and $y \in \R^d.$ We will show that, for some $\Tilde{R}(r)>0,$ there is an $X-$translate of the $r-$patch $\mathcal{P}_{x,r}=X\cap B(x,r)$ that can be found in $B(y,\Tilde{R}(r)).$  Notice that $x \in \Tilde{x} + V_{\Psi(\Tilde{x})} \subseteq \Tilde{x} + [0,1]^d,$ and $y \in \Tilde{y} + [0,1]^d$ for some $\Tilde{x},\Tilde{y} \in \Z^d$ such that $\dist(x,\Tilde{x}) \leq \sqrt{d}$ and $\dist(y,\Tilde{y}) \leq \sqrt{d}$. Also notice that 
    \begin{equation}\label{Px bound}
    \mathcal{P}_{x,r} \subseteq \bigcup_{\substack{a \in \Z^d\\\dist(x,a)< r+\sqrt{d}}} a + V_{\Psi(a)}.
      \end{equation}
  Since $\Psi$ is $R(r)-$repetitive, there exists $\Tilde{w} \in \mathbb{Z}^d$ with $B(\Tilde{w}, r+3\sqrt{d}) \subseteq B(\Tilde{y}, R(r+3\sqrt{d}))$ and for each $z \in \Z^d \cap B(0, r+3\sqrt{d})$
\begin{equation*}
    \Psi(\Tilde{w}+z) = \Psi(\Tilde{x}+z).
      \end{equation*}
Since we take $\dist(x,a) < r+\sqrt{d}$ for each $a$ in (\ref{Px bound}), we know that $a = \Tilde{x} + z$ for some $z \in \Z^d \cap B(0, r+2\sqrt{d}).$ Therefore, 
\begin{equation} \label{eq: Lemma 2.1 main}
\begin{split}
        \mathcal{P}_{x,r} \subseteq \bigcup_{\substack{a \in \Z^d\\\dist(x,a)< r+\sqrt{d}}} a + V_{\Psi(a)}&\subseteq \bigcup_{z \in \Z^d \cap B(0,r+2\sqrt{d}) } \Tilde{x}+z + V_{\Psi(\Tilde{x}+z)}
         =\Tilde{x} - \Tilde{w} + \bigcup_{z \in \Z^d \cap B(0,r+2\sqrt{d}) }  \Tilde{w} + z + V_{\Psi(\Tilde{w}+z)}
        \\& = \Tilde{x} - \Tilde{w} + \bigcup_{z \in \Z^d \cap B(\Tilde{w},r+2\sqrt{d}) } z + V_{\Psi(z)} \subseteq \Tilde{x} - \Tilde{w}+X,
        \end{split}
\end{equation}
and 
\begin{equation} \label{eq: Lemma 2.1 sub}
\begin{split}
    \mathcal{P}_{x,r} &\subseteq \Tilde{x} - \Tilde{w} + \bigcup_{z \in \Z^d \cap B(\Tilde{w},r+2\sqrt{d}) } z + V_{\Psi(z)} \subseteq \Tilde{x}-\Tilde{w}+B(\tilde{y},R(r+3\sqrt{d})).
    \end{split}
\end{equation}
Now take $\Tilde{R}(r) \coloneqq R(r+3\sqrt{d})+\sqrt{d}.$ Then we obtain that $\mathcal{P}_{x,r} - \Tilde{x} + \Tilde{w} \subseteq B(y, \Tilde{R}(r)) \cap X$ using (\ref{eq: Lemma 2.1 main}), \eqref{eq: Lemma 2.1 sub} and that $\dist(y, \Tilde{y}) \leq \sqrt{d}$. 

Setting $$u:=x-\Tilde{x}+\Tilde{w},$$ we note that $u\in X$ and $u-x+\mathcal{P}_{x,r}=\mathcal{P}_{x,r}-\widetilde{x}+\widetilde{w}\subseteq B(y,\Tilde{R}(r))\cap X$. To complete the proof, we verify that $$u-x+\mathcal{P}_{x,r}=\mathcal{P}_{u,r}.$$ The ``$\subseteq$'' inclusion is easy because $u-x+\mathcal{P}_{x,r}\subseteq X$ and $-x+\mathcal{P}_{x,r}\subseteq B(0,r)$. We now verify the ``$\supseteq$'' inclusion. Suppose $v\in \mathcal{P}_{u,r}=X\cap B(u,r)$. Then there exists $\Tilde{v}\in \Z^{d}$ such that $v\in\Tilde{v}+V_{\Psi(\Tilde{v})}$, implying that $\norm{\Tilde{v}-u}<r+\sqrt{d}$. Hence $\norm{\Tilde{v}-\Tilde{w}}\leq \norm{\Tilde{v}-u}+\norm{u-\Tilde{w}}<r+2\sqrt{d}$, which ensures $$\Psi(\Tilde{v})=\Psi(\widetilde{w}+(\Tilde{v}-\Tilde{w}))=\Psi(\Tilde{x}+(\Tilde{v}-\Tilde{w})).$$ Therefore $\Tilde{x}+(\Tilde{v}-\Tilde{w})+V_{\Psi(\Tilde{v})}\subseteq X$, implying $\Tilde{x}-\Tilde{w}+v\in X$. Moreover, we have $\norm{\Tilde{x}-\Tilde{w}+v-x}=\norm{v-u}<r$. Hence $\Tilde{x}-\Tilde{w}+v\in\mathcal{P}_{x,r}$, which rearranges to $v\in -\Tilde{x}+\Tilde{w}+\mathcal{P}_{x,r}=u-x+\mathcal{P}_{x,r}$, as required to prove $u-x+\mathcal{P}_{x,r}\supseteq \mathcal{P}_{u,r}$.
\end{proof}
Lemma \ref{rep from mapping to net} implies that if $\Psi$ is $R(r)-$repetitive, then $X$ is $O(R(r))-$repetitive. We will use this fact later. We now begin the main general construction of a repetitive mapping $\Psi:\Z^d \to [2].$ From now on, the role of the dimension $d \in \N_{\geq 2}$ will be that of a constant. Since the quantity $2^{d}$ will appear often, we introduce the notation $D:=2^{d}.$ We will introduce helpful notation in the next few definitions. 
\begin{define} \label{define: criteria for Psi before} 
\begin{enumerate}[(i)]
   \item\label{def:cubic} A \emph{cubic set} $T \subset \Z^d$ is a set of the form
    \begin{equation*}
        T = \prod_{i=1}^d\{x_i,x_i+1,...,x_i+t-1\}
    \end{equation*}
    for $(x_1,...,x_d) \in \Z^d$ and $t \in \N.$ We call $(x_1,...,x_d)$ the \emph{base point} of $T$, denoted $\emph{bp}(T)$, and $t$ the \emph{sidelength} of $T.$ 
        We refer to the point $(x_1+t-1,...,x_d+t-1) \in \Z^d$ as the \emph{maximal corner} of $T.$ 
        \item \label{natural partition}
    Let $s,t \in \N$ and $(x_1,...,x_d) \in \Z^d$. Let $T \subset \Z^d$ be a cubic set with sidelength $t$ and base point $\emph{bp}(T)=(x_1,...,x_d)$, and suppose $s|t.$ Then we will say the \emph{$s-$natural partition} of $T$ is the lexicographically ordered collection $\mathcal{N}_{T,s}$ of disjoint cubic sets with sidelength $s$ covering $T.$ That is,
    \begin{equation*}
        \mathcal{N}_{T,s}\coloneqq \bigg\{\prod_{i=1}^d \{b_i+y_is,...,b_i+y_is+s-1\}:(y_1,...,y_d) \in \bigg\{0,...,\frac{t}{s}-1\bigg\}^d\bigg\}
    \end{equation*}
   and the elements of $\mc{N}_{T,s}$ are ordered lexicographically with respect to $(y_1,...,y_d).$ 
        \end{enumerate}
    \end{define}
\begin{define} \label{define: criteria for Psi}
   Let $d \in \N_{\geq 2}$, $D \coloneqq 2^d$, $p_0 \coloneqq 0,$ and $(p_n)_{n\in \N} \subseteq \N$ and $(c_n)_{n \in \N}\subseteq \N_{\geq 2}$ be two sequences satisfying for each $n \in \N$ that
   \begin{equation*}
        c_n \leq (2^{p_n-p_{n-1}-1})^{\frac{1}D}.
   \end{equation*}
  For each $n \in \N$, we denote the cubic set $\mathcal{Q}_n^d \coloneqq \{0,...,2^{p_{n-1}}-1\}^d$. 
We call a collection of functions 
   \begin{equation*}
\phi_1^{(n)},...,\phi_{c_n}^{(n)}: \mathcal{Q}_n^d \to [2]
\end{equation*} 
a \emph{palette} (of colours) at level $n.$ For $j \in [c_n]$, the function $\phi_j^{(n)}$ is referred to as the $j^{th}$ colour at level $n$.  
For each $n \in \N$ and $i \in [c_n^D],$ let $H_i^{(n)}$ be the cubic set with sidelength $2^{p_{n-1}+1}$ with base point given by
\begin{equation*}
\emph{bp}(H_i^{(n)})= (2^{p_{n-1}+1}(i-1),0,0,...,0).
\end{equation*}
We have the constraint
     $2c_n^D \leq 2^{p_n - p_{n-1}}$
so that we have enough space to fit $\bigcup_{i=1}^{c_{n}^D} H_i^{(n)}$ in one row along the base of $\mc{Q}_{n+1}^{d}$.

We say that a sequence of palettes $(\phi_{j}^{(n)})_{n\in\N, j\in [c_{n}]}$ is \emph{good} if it satisfies the following two criteria:
\renewcommand{\labelenumi}{(\Alph{enumi})}
 \begin{enumerate}[(A)] 
     \item \label{palette 1} For every $n \in \N_{\geq 2}$, let $\mathcal{Z}_{n}$ be the collection of cubic sets in the $2^{p_{n-2}}-$natural partition of $\mathcal{Q}_n^d.$ Then for every $j\in [c_{n}]$ and $T \in \mathcal{Z}_n$ there exists some $q=q_{j,T} \in [c_{n-1}]$ such that
     \begin{equation*} 
         \phi_j^{(n)}(x) = \phi_{{q}}^{(n-1)}(x-\emph{bp}(T))\qquad \text{for all $x\in T$}.
     \end{equation*} In other words, each colour in the palette at the $n^{th}$ level is made up only of colours from the palette at the $(n-1)^{th}$ level.
     \item \label{palette 2} For every $n \in \N_{\geq 2}$ and $i \in [c_{n-1}^D],$ let $(H_{i,l}^{(n-1)})_{l \in [D]}$ be the ordered natural $2^{p_{n-2}}$-partition of the set $H_i^{(n-1)}$ into $D$ cubic sets. Let $(a_{i}^{(n-1)})_{i\in [c_{n-1}^{D}]}$ be the lexicographical ordering of  $[c_{n-1}]^D$. Then for each $j \in [c_n]$ we have
     \begin{equation*} 
         \phi_j^{(n)}(x) = \phi^{(n-1)}_{a_{i,k}}(x-\emph{bp}(H_{i,k}^{(n-1)}))\qquad\text{for all $x\in H_{i,k}^{(n-1)}$, $i\in[c_{n-1}^{D}]$ and $k\in [D]$,}
     \end{equation*}where $a_{i,k}=(a_{i}^{(n-1)})_{k}\in [c_{n-1}]$.
     In other words, $\bigcup_{i=1}^{c_{n-1}^D}H_{i}^{(n-1)}$ positioned `along the base' of each colour at the $n^{th}$ level `contains' every ordered $D-$combination of colours at the $(n-1)^{th}$ level.
 \end{enumerate}
\end{define} 
\paragraph{A quick aside about colours and palettes}
A helpful way to interpret the concepts in Definition \ref{define: criteria for Psi} is to consider each colour $\phi_i^{(n)}$ as an element of $\{1,2\}^{\mathcal{Q}_n^d}.$ Pictorially, one may consider drawing $\phi_i^{(n)}$ by considering each unit cube with integer vertices whose bottom-left corner is an element of ${\mathcal{Q}_n^d},$ and then putting a $1$ (or a light shade) or a $2$ (or a dark shade) into each cube dependent on the value of $\phi_i^{(n)}$ on its bottom-left corner. Consider the example in Figure \ref{fig: example} with $d = 2$ and $p_{n-1} = 2.$ Each $4\times 4$ `tile' represents a colour and we call the collection of tiles the palette at the $n^{th}$ level. At this level the sidelength of each tile is $2^{p_{n-1}}.$ In this sense, $(p_n)_{n \in \N}$ controls the `complexity' of each colour, that is, the shade of each tile (see Claim \ref{bij prop} for a precise definition) can be controlled more finely if $p_{n-1}$ is large. When constructing at a sequence of good palettes as in Definition \ref{define: criteria for Psi}, point \ref{palette 1} demands that we use copies of the colours at the $n^{th}$ level only to tile $\mathcal{Q}_{n+1}^d;$ in other words, some combination of the tiles in our example (with repetition allowed) is glued together into a larger tile (of sidelength $2^{p_n})$ to form each of the colours at level $n+1.$ Since the number of tiles at the $n^{th}$ level is $c_n,$ the sequence $(c_n)_{n \in \N}$ controls how many colours in our palette we have to `paint' the colours at the level above, hence justifying the use of artistic terminology in our construction. 

\begin{figure}[h] 
\centering
\scalebox{1}{\begin{tikzpicture}
\centering
\fill[red,fill opacity = 0.1] (0,0) rectangle (0.8,0.8);
\fill[red,fill opacity = 0.1,xshift = -4cm] (0,0) rectangle (0.8,0.8);
\fill[red,fill opacity = 0.7] (0,0) rectangle (0.2,0.2);
\fill[red,fill opacity = 0.7,xshift = -4cm] (0,0) rectangle (0.2,0.2);
\fill[red,fill opacity = 0.7] (0.6,0) rectangle (0.8,0.2);
\fill[red,fill opacity = 0.7,xshift = -4cm] (0.4,0) rectangle (0.6,0.2);
\fill[red,fill opacity = 0.7] (0.6,0.6) rectangle (0.8,0.8);
\fill[red,fill opacity = 0.7] (0.4,0.2) rectangle (0.6,0.4);
\fill[red,fill opacity = 0.7] (0,0.4) rectangle (0.4,0.8);
\draw[step=0.2cm,black,very thick,xshift = -4cm] (0,0) grid (0.8,0.8);
\draw[black,very thick] (4,0) rectangle (4.4,0.4);
\fill[red,fill opacity = 0.7] (4,0) rectangle (4.4,0.4);
\draw[black,very thick, yshift = 0.6cm] (4,0) rectangle (4.4,0.4);
\fill[red,fill opacity = 0.1, yshift = 0.6cm] (4,0) rectangle (4.4,0.4);
\draw[<->,xshift = 0.1cm] (4.4,0.2) -- (5,0.2);
\draw[<->,yshift = 0.6cm,xshift = 0.1cm] (4.4,0.2) -- (5,0.2);
\draw[<->] (0,1.1) -- (0.8,1.1);
\draw[<->,xshift = -4cm] (0,1.1) -- (0.8,1.1);
\node at (0.4,1.3) {$2^{p_{n-1}}$};
\node at (-3.6,1.3) {$2^{p_{n-1}}$};
\draw[<->] (0.6,-0.1) -- (0.8,-0.1);
\draw[<->,xshift = -4cm] (0.6,-0.1) -- (0.8,-0.1);
\node at (0.7,-0.4) {$1$};
\node at (-3.3,-0.4) {$1$};
\node at (5.3,0.8) {$1$};
\node at (5.3,0.2) {$2$};
\node at (-5,0.5) {$\phi_1^{(n)}$};
\node at (-4.3,-0.3) {$(0,0)$};
\filldraw [gray] (-4,0) circle (2pt);
\draw[step=0.2cm,black,very thick] (0,0) grid (0.8,0.8);
\node at (-1,0.5) {$\phi_{c_n}^{(n)}$};
\node at (-0.3,-0.3) {$(0,0)$};
\filldraw [gray] (0,0) circle (2pt);
\node[draw=none] (ellipsis1) at (-2,0.4) {$\hdots$};
\end{tikzpicture}}
\caption{Example of a palette of colours}\label{fig: example}
\end{figure}
In this paper, we will reserve the final colour in each palette, $\phi_{c_n}^{(n)},$ to be an `approximation' of a density $\rho:[0,1]^d \to [\frac{4}{3},\frac{5}{3}]$. This is more naturally noticed in this tiling-style description of colours - namely, to create the $n^{th}$ `approximation' of $\rho$ we put, into each unit cube with integer vertices and bottom-left corner in $\mathcal{Q}_n^d$, the average value of $2^{p_{n-1}}\rho$ over that cube, rounded to the nearest integer. This ensures the eventual Delone set we construct has the property of \emph{encoding} $\rho$ (see Definition \ref{defn: X encoding rho}), which in conjunction with Theorem \ref{prop:non-rectifiable non-realisable} gives us the desired property of non-rectifiability if we choose $\rho$ to be a non-realisable density. Finally, the purpose of point \ref{palette 2} of Definition \ref{define: criteria for Psi} is to ensure that the eventually constructed Delone set is repetitive (see the proof of Theorem \ref{intermediate thm}). 

The next two lemmas are preliminary to Theorem \ref{main construction}, the major result of this chapter.
\begin{lemma} \label{Lemma: colours inside colours}
     Let the sequences $(p_n)_{n \in \N \cup \{0\}}$ and $(c_n)_{n \in \N}$ be as in Definition \ref{define: criteria for Psi}, and assume a sequence of palettes $(\phi_{j}^{(n)})_{n\in\N, j\in [c_{n}]}$ obeys point (\ref{palette 1}) from Definition \ref{define: criteria for Psi}. Let $n \in \N$, $m \in [n] \cup \{0\}$ and $T$ be a cubic set in the $2^{p_{m}}-$natural partition of $\mc{Q}^{d}_{n+1}$. Then there exists some $j \in [c_{m+1}]$ such that for each $x \in T$
    \begin{equation*}
        \phi_1^{(n+1)}(x) = \phi_j^{(m+1)}(x-\emph{bp}(T)).
    \end{equation*}
\end{lemma}
\begin{proof}
    Fix $n \in \N.$ The lemma is true for $m = n$ since the only cubic set in the natural $2^{p_{n}}-$partition of $\mathcal{Q}_{n+1}^{d}$ has base point $0$, so we may take $j=1$. Now suppose the lemma is true for $m \in [n]$ and we will show it is true for $m-1.$ Let $T$ be a cubic set in the $2^{p_{m-1}}-$natural partition of $\mathcal{Q}_{n+1}^d$, and let $\widetilde{T}$ be the cubic set in the $2^{p_{m}}-$natural partition of $\mathcal{Q}_{n+1}^d$ containing $T.$ By the induction hypothesis, there exists $j_{0}\in [c_{m+1}]$ such that for each $x \in \widetilde{T}$ 
    \begin{equation} \label{eqn: colours inside colours 1}
        \phi_1^{(n+1)}(x) = \phi_{{j_{0}}}^{(m+1)}(x-\text{bp}(\widetilde{T})).
    \end{equation} Thus, the induction hypothesis gives us a formula for $\phi_{1}^{(n+1)}(x)$ when $x$ belongs to the cubic set $\widetilde{T}$. In the induction step, we manipulate \eqref{eqn: colours inside colours 1} to obtain a similar representation of $\phi_{1}^{(n+1)}(x)$ when $x$ belongs to the smaller cubic set $T\subset \widetilde{T}$. Acknowledging that $T - \text{bp}(\widetilde{T})$ is a set in the $2^{p_{m-1}}-$natural partition of $\mathcal{Q}_{m+1}^d$, we let $j:=q_{j_{0},T-\text{bp}(\widetilde{T})}$ from \eqref{palette 1} so that for each $y \in T-\text{bp}(\widetilde{T})$
    \begin{equation} \label{eqn: colours inside colours 2}
        \phi_{{j_{0}}}^{(m+1)}(y) = \phi_j^{(m)}(y-\text{bp}(T-\text{bp}(\widetilde{T}))) = \phi_j^{(m)}(y-\text{bp}(T)+\text{bp}(\widetilde{T})).
    \end{equation} Putting equations (\ref{eqn: colours inside colours 1}) and (\ref{eqn: colours inside colours 2}) together, writing $x = y+\text{bp}(\widetilde{T}),$ we get for each $x \in T$ that $\phi_1^{(n+1)}(x) = \phi_j^{(m)}(x-\text{bp}(T)),$ as required.
\end{proof}
\begin{lemma} \label{Delone spans space}
    Let $p_0 \coloneqq 0$ and $(p_n)_{n \in \N}\subseteq \N$ be a sequence satisfying $p_{n}-p_{n-1}\geq 2$ for each $n \in \N$. For each $n \in \N$, let $H_1^{(n)}$ be as defined in Definition \ref{define: criteria for Psi}, and let $s_n \coloneqq -\left(\sum_{j=1}^{n-1}2^{p_j}\right)(1,1,...,1)$. Then for each $n \in \N,$ we have that $s_n+H_1^{(n)} \subset s_{n+1}+H_1^{(n+1)}$ and
    \begin{equation*}
        \bigcup_{n = 1}^\infty \left(s_n+H_1^{(n)}\right) = \Z^d.
    \end{equation*} 
\end{lemma}
\begin{proof}
As in the notation of Definition \ref{define: criteria for Psi}, for each $n \in \N$ notice that $H_1^{(n)} \subset H_{1,1}^{(n+1)}.$ Moreover, $H_{1,1}^{(n+1)} \subset H_1^{(n+1)} -2^{p_n}(1,1,...,1) = H_1^{(n+1)}+s_{n+1}-s_n.$ These two observations give us that $H_1^{(n)}+s_n \subset H_1^{(n+1)}+s_{n+1}.$ 
Now we show that $\bigcup_{n=1}^\infty (s_n+H_1^{(n)}) = \Z^d.$ 
Observe that for each $n \in \N,$ 
the base point of the cubic set $s_{n}+H_{1}^{(n)}$ is 
\begin{equation*}
    s_n=-\left(\sum_{j=1}^{n-1}2^{p_j}\right)(1,1,...,1),\qquad \text{ where }-\left(\sum_{j=1}^{n-1}2^{p_j}\right)\xrightarrow{n\rightarrow\infty} -\infty,
\end{equation*}
whilst the maximal corner (see Definition~\ref{define: criteria for Psi before}\eqref{def:cubic}) of the cubic set $s_{n}+H_{1}^{(n)}$ is 
\begin{equation*}
    s_n + (2^{p_{n-1}+1}-1)(1,1,...,1) =\left(-\left(\sum_{j=1}^{n-1}2^{p_j}\right) + 2^{p_{n-1}+1}-1\right)(1,1,...,1).
\end{equation*}
Therefore, the proof will be complete if we show that
\begin{equation*}
        -\left(\sum_{j=1}^{n-1}2^{p_j}\right) + 2^{p_{n-1}+1} \xrightarrow{n\rightarrow\infty} \infty.
    \end{equation*}
    Indeed, since $p_n-p_{n-1} \geq 2$ for each $n \in \N$, we have for each $n \in \N_{\geq 2}$ that
    \begin{equation*}
        \frac{1}{2^{p_{n-1}}} \sum_{j=1}^{n-1}2^{p_j} =  \sum_{j=1}^{n-1}2^{p_j-p_{n-1}} \leq  \sum_{j=1}^{n-1}2^{2(j-(n-1))} \leq  \sum_{j=0}^{\infty}2^{-2j} = \frac{4}{3}.
    \end{equation*}
    Therefore, 
    \begin{equation*}
        -\left(\sum_{j=1}^{n-1}2^{p_j}\right) + 2^{p_{n-1}+1} \geq 2^{p_{n-1}+1}-\frac{4}{3}2^{p_{n-1}} = \frac{2}{3}2^{p_{n-1}} \xrightarrow{n\rightarrow\infty} \infty.
    \end{equation*}
\end{proof}
The following theorem is the main result of this section. It shows how a good sequence of palettes, as in Definition \ref{define: criteria for Psi}, can be manipulated into a repetitive function $\Psi:\Z^d \to [2].$ Moreover, every colour in the the collection of palettes can be `identified' in some cubic section of $\Psi.$  
\begin{thm} \label{main construction}
    Let $d \in \N_{\geq 2}$, $D \coloneqq 2^d$, $p_0 \coloneqq 0$, $(p_n)_{n\in \N} \subseteq \N$ and $(c_n)_{n\in \N} \subseteq \N_{\geq 2}$ be non-decreasing sequences satisfying for each $n \in \N$ that
    \begin{equation} \label{constraint cn pn}
        c_n \leq (2^{p_n-p_{n-1}-1})^{\frac{1}D} \qquad \emph{and} \qquad p_n-p_{n-1} \geq 2. 
    \end{equation}
 Suppose there exists a sequence of palettes $(\phi_{j}^{(n)})_{n\in\N, j\in [c_{n}]}$ which are good according to Definition \ref{define: criteria for Psi}. Then there exists a well-defined function $\Psi:\Z^d \to [2]$ such that
 \begin{enumerate}[(i)]
     \item \label{eq: repetitivity function}  $\Psi$ is repetitive with repetitivity function 
     \begin{equation*}
         R(r) = \sqrt{d}\;2^{p_n} \qquad \text{whenever}\; 2^{p_{n-2}}<2r\leq 2^{p_{n-1}}
     \end{equation*}
 \item \label{condition: every colour found somewhere} for every $n \in \N$ and every $j \in [c_n],$ there exists a cubic set $T_{n,j} \subset \Z^d$ with sidelength $2^{p_{n-1}}$ and base point $\emph{bp}(T_{n,j})$ such that 
 \begin{equation*}
     \Psi(x) = \phi_j^{(n)}(x-\emph{bp}(T_{n,j}))
 \end{equation*} for all $x \in T_{n,j}.$ 
 \end{enumerate}
\end{thm}
\begin{proof} 
For each $n \in \N$, let $H_1^{(n)}$ be as defined in Definition \ref{define: criteria for Psi}, and let $s_n \coloneqq -\left(\sum_{j=1}^{n-1}2^{p_j}\right)(1,1,...,1)$. By Lemma \ref{Delone spans space}, we have
    \begin{equation*}
        \bigcup_{n = 1}^\infty \left(s_n+H_1^{(n)}\right) = \Z^d.
    \end{equation*} With this in mind, we define the function $\Psi:\Z^d\rightarrow [2]$ as follows. For each $x \in \Z^d,$ choose $N$ to be the smallest natural number such that $x \in s_N+H_1^{(N)}.$ Then we prescribe
    \begin{equation*}
        \Psi(x) \coloneqq \phi_1^{(N+1)}(x-s_N).
    \end{equation*}
    This is a well-defined function since the domain of $\phi_1^{(N+1)}$ contains $H_1^{(N)}.$ We now need the following claim:

    \begin{claim} \label{claim:phiNM}
        If $x \in s_M + H_1^{(M)}$ for some $M \in \N$ with $M \geq N,$ then $\Psi(x) = \phi_1^{(M+1)}(x-s_M).$ 
    \end{claim}
    \begin{proof}
    We refer to the notation from Definition \ref{define: criteria for Psi}. 
        We know that $\Psi(x) = \phi_1^{(N+1)}(x-s_N)$ where $N$ is minimal satisfying $x \in s_N + H_1^{(N)}.$ Notice by Lemma \ref{Delone spans space} that $x \in s_N + H_1^{(N)}\subseteq s_M + H_1^{(M)},$ so $\phi_1^{(M+1)}(x-s_M)$ is well-defined. 
        Notice that \begin{equation} \label{eqn: periodicity sn}
               s_{N} -s_M = -\left(\sum_{j=1}^{N-1}2^{p_j}-\sum_{j=1}^{M-1}2^{p_j}\right)(1,1,...,1) = \sum_{j=N}^{M-1}2^{p_j} (1,1,...,1)\in 2^{p_N}\Z^d.
           \end{equation} 
           We now apply Lemma \ref{Lemma: colours inside colours}. Let the cubic set $T$ in the $2^{p_N}-$natural partition of 
$\mc{Q}^{d}_{M+1}$ be so that $\text{bp}(T) = s_N-s_M \in 2^{p_N}\Z^d$. Note that $x-s_N \in H_1^{(N)}$ and so $x-s_M \in H_1^{(N)}+s_N-s_M \subset T$. Therefore, for some $j \in [c_{N+1}]$
\begin{equation*}
    \phi_1^{(M+1)}(x-s_M) = \phi_j^{(N+1)}(x-s_M - \text{bp}(T)) = \phi_j^{(N+1)}(x-s_M - (s_N-s_M)) =\phi_j^{(N+1)}(x-s_N).
\end{equation*}
Finally, we use that $x-s_N \in H_1^{(N)}$ and so by point \eqref{palette 2}, $x-s_N \in H_{1,k}^{(N)}$ for some $k \in [D]$ and
\begin{equation*}
    \phi_j^{(N+1)}(x-s_N) = \phi_{a_{1,k}}^{(N)}(x-s_N-\text{bp}(H_{1,k}^{(N)})) = \phi_1^{(N+1)}(x-s_N)=\Psi(x).
\end{equation*}
    \end{proof}
    We now show point (\ref{condition: every colour found somewhere}) holds. Fix $n \in \N$ and $j \in [c_n].$ Using the lexicographic notation from Definition \ref{define: criteria for Psi}, there exists some $i \in [c_{n}^D]$ and $k \in [D]$ such that $(a_{i}^{(n-1)})_{k} = a_{i,k} =j.$ Then take $T_{n,j} \coloneqq H_{i,k}^{(n)}+s_{n+1} \subset H_1^{(n+1)}+s_{n+1}.$ For each $x \in T_{n,j}$ we have
    \begin{equation*}
        \Psi(x) = \phi_1^{(n+2)}(x-s_{n+1}) = \phi_1^{(n+1)}(x-s_{n+1})= \phi_{a_{i,k}}^{(n)}(x-s_{n+1}-\text{bp}(H_{i,k}^{(n)}))= \phi_j^{(n)}(x-\text{bp}(T_{n,j})).
    \end{equation*}
 The first equality uses Claim \ref{claim:phiNM} and the definition of $\Psi$ using that $x-s_{n+1} \in H_1^{(n+1)}.$ The second and third equalities use point (\ref{palette 2}) of Definition \ref{define: criteria for Psi}, acknowledging that $x-s_{n+1} \in H_{i,k}^{(n)} \subset H_{1,1}^{(n+1)}$ and $\text{bp}(H_{1,1}^{(n+1)})=0$. 
It remains to show point (\ref{eq: repetitivity function}). We will first prove the following claim:
    \begin{claim} \label{claim: psi partition into colours}
        Let $n \in \N.$ Let $\mathcal{W}_n$ be the partition of $\Z^d$ into cubic sets with sidelength $2^{p_{n-1}}$ and base points in the set $s_{n}+2^{p_{n-1}}\Z^d.$ Then for each $W \in \mathcal{W}_n,$ there exists $j \in [c_n]$ such that for each $x \in W$
        \begin{equation*}
            \Psi(x) = \phi_j^{(n)}(x-\emph{bp}(W)).
        \end{equation*}
    \end{claim}
    \begin{proof}
        There exists some $N \geq n$ such that $W \subseteq H_1^{(N)}+s_N$ by Lemma \ref{Delone spans space}. Therefore, by Claim \ref{claim:phiNM}, for each $x \in W$
        \begin{equation} \label{eqn: claim 1}
           \Psi(x)= \phi_1^{(N+1)}(x-s_N).
        \end{equation}
       Notice $W-s_N$ is in the $2^{p_{n-1}}-$natural partition of $H_1^{(N)}.$ Indeed, we have that $W-s_N \subseteq H_1^{(N)}$ and $W-s_N$ is a cubic set of sidelength $2^{p_{n-1}},$ so it suffices to show that $\text{bp}(W-s_N) \in 2^{p_{n-1}}\Z^d.$ We know that $W \in \mathcal{W}_n,$ so $\text{bp}(W) \in s_n+2^{p_{n-1}}\Z^d,$ and so 
       \begin{equation*}
           \text{bp}(W-s_N) = \text{bp}(W)-s_N \in s_n-s_N + 2^{p_{n-1}}\Z^d \subseteq 2^{p_{n-1}}\Z^d,
           \end{equation*} where for the last set inclusion we use (\ref{eqn: periodicity sn}).
        Therefore, by Lemma \ref{Lemma: colours inside colours}, there exists some $j \in [c_n]$ so that for each $y \in W-s_N$
        \begin{equation} \label{eqn: claim 2}
            \phi_1^{(N+1)}(y) = \phi_j^{(n)}(y-\text{bp}(W-s_N))=\phi_j^{(n)}((y+s_N)-\text{bp}(W)).
        \end{equation} Setting $y = x-s_N$ in (\ref{eqn: claim 2}),  we get from (\ref{eqn: claim 1}) that
            $\Psi(x) = \phi_j^{(n)}(x-\text{bp}(W)),$ as required.
    \end{proof} 
    We return to the proof of point (\ref{eq: repetitivity function}). Fix $x,y \in \Z^d$ and $r>0$ such that $2^{p_{n-2}}<2r \leq 2^{p_{n-1}}$ for some $n \in \N_{\geq 2}.$ There is a cubic set $U$ of sidelength $2^{p_{n-1}+1}$ with base point in $2^{p_{n-1}}\Z^{d}$ such that
    \begin{equation*}
        B(x,r)\cap \Z^{d}\subseteq U.
    \end{equation*}
    We consider the $2^{p_{n-1}}-$naturally partition of $U$ into $D \coloneqq 2^d$ cubic sets $U_1,...,U_D$, ordered according to Definition~\ref{define: criteria for Psi}\eqref{def:cubic} and with base points in $2^{p_{n-1}}\Z^d.$ By Claim \ref{claim: psi partition into colours}, there exists $(a_1,a_2,...,a_D) \in [c_{n}]^D$ such that for each $i \in [D]$ and $z \in U_i$
    \begin{equation} \label{eqn:8}
        \Psi(z)= \phi_{a_i}^{(n)}(z-\text{bp}(U_i)).
    \end{equation}
    By Claim \ref{claim: psi partition into colours}, since $y$ belongs to a cubic set $W_y \in \mathcal{W}_{n+1},$ there exists some $l \in [c_{n+1}]$ so that for each $z \in W_y$ 
    \begin{equation} \label{eqn: 2pn around y}
        \Psi(z) = \phi_l^{(n+1)}(z-\text{bp}(W_y)).
    \end{equation}
     By point (\ref{palette 2}) of Definition \ref{define: criteria for Psi}, there exists some $j \in [c_n^D]$ so that the cubic set $H_{j}^{(n)}=\bigcup_{i=1}^{D}H_{j,i}^{(n)}$ satisfies for each $i \in [D]$ and $z \in H_{j,i}^{(n)}$
    \begin{equation}\label{eqn: the right Dcolouring}
        \phi_l^{(n+1)}(z) = \phi_{a_i}^{(n)}(z-\text{bp}(H_{j,i}^{(n)})).
    \end{equation}
    Combining equations (\ref{eqn: 2pn around y}) and (\ref{eqn: the right Dcolouring}), we get for each $z \in H_{j,i}^{(n)}+\text{bp}(W_y) \subset W_y$ that 
    \begin{equation} \label{eqn:11}
        \Psi(z) = \phi_{a_i}^{(n)}(z-\text{bp}(H_{j,i}^{(n)})-\text{bp}(W_y)).
    \end{equation}
    Equations \eqref{eqn:8} and \eqref{eqn:11} state that for the two cubic sets $U$ and $\text{bp}(W_{y})+H_{j}^{(n)}$ of sidelength $2^{p_{n-1}+1}$ we have that 
    \begin{equation*}
        \Psi(z)=\Psi(z-\text{bp}(W_{y})-\text{bp}(H_{j}^{(n)})+\text{bp}(U)) \qquad \text{for all $z\in \text{bp}(W_{y})+H_{j}^{(n)}$.}
    \end{equation*}
    Therefore, setting $w=\text{bp}(W_{y})+\text{bp}(H_{j}^{(n)})+x-\text{bp}(U)$ we have that 
    \begin{equation*}
    \begin{split}
    B(w,r)\cap \Z^{d}=w-x+B(x,r)\cap\Z^{d}&=\text{bp}(W_{y})+\text{bp}(H_{j}^{(n)})-\text{bp}(U)+B(x,r)\cap\Z^{d}\\
    &\subseteq\text{bp}(W_{y})+\text{bp}(H_{j}^{(n)})-\text{bp}(U)+U=\text{bp}(W_{y})+H_{j}^{(n)},
    \end{split}
    \end{equation*}
     and for all $z\in \Z^{d}\cap B(0,r)$
    \begin{equation*}
        \Psi(w+z)=\Psi(w+z-\text{bp}(W_{y})-\text{bp}(H_{j}^{(n)})+\text{bp}(U))=\Psi(x+z). 
    \end{equation*}
   Moreover, since $B(w,r)\cap \Z^{d} \subset H_{j}^{(n)}+\text{bp}(W_y)\subset W_y$, we have that $\dist({y,z}) \leq \diam(W_y) = \sqrt{d}\; 2^{p_n}$ for each $z \in B(w,r),$ so $B(w,r) \subseteq B(y,\sqrt{d} \; 2^{p_n}).$ Therefore, $\Psi$ is repetitive, and has repetitivity function given by $R(r) = \sqrt{d}\; 2^{p_n}$ for $2^{p_{n-2}}<2r\leq 2^{p_{n-1}}$.
\end{proof}

\section{Encoding a non-realisable density in a sequence of palettes} 
In this section, for a given density $\rho:[0,1]^d \to [\frac{4}{3},\frac{5}{3}],$ we construct a good sequence of palettes in the sense of Definition~\ref{define: criteria for Psi} which approximates $\rho$ in a specific way. 
\begin{thm} \label{intermediate thm} Let $d\in \N_{\geq 2}$, $D\coloneqq 2^{d}$ and $\rho\colon [0,1]^{d}\to \sqbr{\frac{4}{3},\frac{5}{3}}$ be a measurable function. Let $p_0 =0$, $c_1 = 3$ and the sequences $(p_n)_{n \in \N} \subseteq \N$ and $(c_n)_{n \in \N} \subseteq \N_{\geq 3}$ satisfy the following conditions for every $n \in \N_{\geq2}$:
\begin{equation}\label{cn tn pn}
\begin{split}
&c_n \leq \min\{(2^{p_n-p_{n-1}-1})^{\frac{1}D}, D(c_{n-1}-2)^{D+1}+2\} \qquad \emph{and} \qquad \sum_{i=1}^{\infty}c_i^D2^{-d(p_{i}-p_{i-1}-1)}\leq \frac{1}{3}.
\end{split} 
\end{equation}
For each $n \in \N_{\geq 2}$, let 
    $\mathcal{U}_n \coloneqq \{\prod_{i=1}^{d}[(m_{i}-1)2^{p_{n-2}},m_{i}2^{p_{n-2}}):(m_{1},\ldots,m_{d})\in [2^{p_{n-1}-p_{n-2}}]^d\}$ be the standard partition of $[0,2^{p_{n-1}})^{d}$ into $2^{d(p_{n-1}-p_{n-2})}$ half-open cubes of side-length $2^{p_{n-2}}$.
    Then there exists a good sequence of palettes $(\phi_{j}^{(n)})_{n\in\N, j\in [c_{n}]}$, according to Definition~\ref{define: criteria for Psi}, such that for each $n \in \N_{\geq 2}$ there exists $\widetilde{\mathcal{U}_n} \subset \mathcal{U}_n$ with $\abs{\widetilde{\mathcal{U}_n}}=2^{d(p_{n-1}-p_{n-2})}-Dc_{n-1}^D$ and for every $U\in\widetilde{\mathcal{U}_n}$
    \begin{equation} \label{rho bk}
        \abs{\sum_{x\in \Z^{d}\cap U}\phi_{c_{n}}^{(n)}(x)-2^{dp_{n-1}}\int_{2^{-p_{n-1}}U}\rho\,d\leb} \leq \frac{2^{dp_{n-2}}}{c_{n-1}-2}.
    \end{equation}
\end{thm}
\begin{proof} 
Recall the objects $\mc{Q}_{n}^{d}$ and $H_{i}^{(n)}$ for $n\in\N$ and $i\in[c_{n}]$ determined by the sequences $(p_{n})_{n=1}^{\infty}$ and $(c_{n})_{n=1}^{\infty}$ in Definition~\ref{define: criteria for Psi}. Note that the first inequality of \eqref{cn tn pn} imples $p_{n}-p_{n-1}>D+1$ for all $n\geq 2$. For each $n \in \N_{\geq 2}$, let 
\begin{equation}\label{eq:tn}
t_n\coloneqq 2^{d(p_{n-1}-p_{n-2})}-Dc_{n-1}^D
\end{equation}
and note, using the first inequality of \eqref{cn tn pn} and $p_{n-1}-p_{n-2}>1$, that $t_{n}\geq 1$.
For each $n \in \N_{\geq 2}$, let $ \widetilde{\mathcal{U}}_n = \{\widetilde{U_{n,k}}\}_{k \in [t_n]}\subset\mathcal{U}_n$ be the subset of $\mc{U}_{n}$ (equipped with an arbitrary ordering) formed by the $t_n$ half-open cubes not intersecting $[0,2^{p_{n-2}+1}c_{n-1}^D)\times [0,2^{p_{n-2}+1})^{d-1}\supseteq \bigcup_{i=1}^{c_{n-1}^{D}}H_{i}^{(n-1)}$. 
Now we pick the first palette determined by
\begin{equation*}
    \phi_1^{(1)}((0,0,...,0)) = 1, \qquad  \phi_2^{(1)}((0,0,...,0)) = 2 \qquad \text{and} \qquad \phi_3^{(1)}((0,0,...,0)) = 1.
\end{equation*} This trivially satisfies the points $(\ref{palette 1})$ and $(\ref{palette 2})$ from Definition \ref{define: criteria for Psi}, since these only apply to $\phi^{(n)}_{j}$ with $n\geq 2$. 
Fix $n \in \N_{\geq 2}$ and suppose that for each $l \in [n-1],$ the $l^{th}$ palette $(\phi^{(l)}_j)_{j \in [c_{l}]}$ has been determined and satisfies points $(\ref{palette 1})$ and $(\ref{palette 2})$ from Definition \ref{define: criteria for Psi}. We construct the $n^{th}$ palette $(\phi^{(n)}_j)_{j \in [c_{n}]}$ as follows. First, ensure that it satisfies point $(\ref{palette 2})$ from Definition \ref{define: criteria for Psi}. For each $j \in [c_n]$, this determines the value of $\phi^{(n)}_j$ on $\mathcal{Q}_n^d \setminus (\bigcup \widetilde{\mathcal{U}_n} \cap \Z^d)$. 
Now let $\mathcal{M}_n \coloneqq [c_{n-1}-1]^{t_n}.$ Notice 
\begin{equation*} \label{eqn:palette shades} 
    \Sigma(\mc{M}_{n}):=\bigg\{\sum_{k =1}^{t_n}M_k :\; M=(M_k)_{k \in [t_n]} \in \mathcal{M}_n\bigg\} = \{t_n,t_n+1,...,t_n(c_{n-1}-1)-1,t_n(c_{n-1}-1)\}.
\end{equation*} 
Our aim is now to pick out a subset of $\Sigma(\mc{M}_{n})$, containing precisely $c_{n}-1$ elements, in such a way that this subset approximates as well as possible any element of $\Sigma(\mc{M}_{n})$. We also want the subset to contain the maximium and minimum elements of $\Sigma(\mc{M}_{n})$. To achieve this, we begin by nominating an approximate `step size' \begin{equation} \label{cn hn}
    h_{n}:=\bigg\lceil\frac{t_{n}(c_{n-1}-2)}{c_{n}-2}\bigg\rceil.
\end{equation}
Note this sequence is well-defined by the first condition of (\ref{cn tn pn}). The idea is to start at the minimum element of $\Sigma(\mc{M}_{n})$ and then make $c_{n}-1$ steps, each of size either $h_{n}-1$ or $h_{n}$ until we get to the maximum of $\Sigma(\mc{M}_{n})$, choosing an appropriate distribution of $(h_{n}-1)$'s and $h_{n}$'s for the steps. We make this formal as follows: Note that the definition of $h_{n}$ allows us to find $\alpha=\alpha_{n}\in[c_{n}-2]\cup\set{0}$ such that $$\alpha(h_{n}-1)+(c_{n}-2-\alpha) h_{n}=t_{n}(c_{n-1}-2).$$ Then we nominate the subset $$\set{t_{n}+\min\set{k,\alpha}(h_{n}-1)+\max\set{k-\alpha,0}h_{n}\colon k\in [c_{n}-2]\cup\set{0}}\subseteq \Sigma(\mc{M}_{n}).$$
For each $i\in [c_{n}-1]$ we choose $M^{(n,i)}\in \mathcal{M}_n$ such that 
\begin{equation*}
    \sum_{k=1}^{t_n} M^{(n,i)}_k=t_{n}+\min\set{i-1,\alpha}(h_{n}-1)+\max\set{i-1-\alpha,0}h_{n},
\end{equation*}
noting that
\begin{multline} \label{eqn: Mn difference}
    \sum_{k=1}^{t_n} M^{(n,1)}_k=t_{n},\qquad \sum_{k=1}^{t_n} M^{(n,c_{n}-1)}_k=t_{n}(c_{n-1}-1)
    \quad \text{and}\;\; \sum_{k=1}^{t_n} M^{(n,i+1)}_k-\sum_{k=1}^{t_n} M^{(n,i)}_k \in \{h_n-1,h_n\}
\end{multline} for all $i\in[c_{n}-2].$
Moreover, we can ensure for each $i \in [c_{n}-1] \setminus \{1\}$ and $k\in[t_{n}]$ that 
\begin{equation}\label{eq:Mincreasing}
 M^{(n,i)}_{k}-M^{(n,i-1)}_{k}\geq 0.
\end{equation}
We also note that
\begin{equation}\label{eq:first_and_last}
    M^{(n,1)}_{k}=1 \qquad \text{and}\qquad M^{(n,c_{n}-1)}_{k}=c_{n-1}-1 \qquad \text{for all $k\in [t_{n}]$.}
\end{equation}
Now suppose $j \in [c_{n}-1]$. For each $k \in [t_n]$, we assign the colour $M \coloneqq M^{(n,j)}_k$ from the $(n-1)^{th}$ palette to $\Z^d \cap \widetilde{U_{n,k}}.$ That is, for each $x \in \widetilde{U_{n,k}} \cap \Z^d$
\begin{equation}\label{eq:construction}
    \phi^{(n)}_{j}(x)\coloneqq \phi^{(n-1)}_{M}(x-\text{bp}(\Z^d \cap\widetilde{U_{n,k}})).
\end{equation} 
Notice that this construction obeys point $(\ref{palette 1})$ from Definition \ref{define: criteria for Psi}. Now consider the remaining case of $j=c_{n}$. In other words, we construct the last colour of the $n^{th}$ palette $\phi_{c_n}^{(n)}.$ Before we proceed, we will require the following claim. 
\begin{claim} \label{bij prop}
Let $m \in [n-1].$ For each $i \in [c_m]$, let the \emph{shade} of colour $i$ at level $m$ be 
    \begin{equation*} \label{bij}
        b_{i}^{(m)} \coloneqq \frac{1}{2^{dp_{m-1}}} \sum_{x \in \mathcal{Q}_m^d} \phi^{(m)}_i(x).
    \end{equation*} 
     Then the following two points hold: 
    \renewcommand{\labelenumi}{\theenumi.}
    \begin{enumerate}[(i)]
        \item \label{bij prop:point 1}
$b_{i}^{(m)}-b_{i-1}^{(m)} \leq \left(\prod_{l=1}^m h_l\right)2^{-dp_{m-1}}\qquad  \emph{for each}\; i \in [c_{m}-1] \setminus \{1\}$
        \item \label{bij prop:point 2} $b_{1}^{(m)} \leq 1 + \sum_{i=1}^{m-1}c_i^D2^{-d(p_i-p_{i-1}-1)} \leq \frac{4}{3}$ \;\; \emph{and} \;\; $b_{c_{m}-1}^{(m)} \geq 2-\sum_{i=1}^{m-1}c_i^D2^{-d(p_i-p_{i-1}-1)} \geq \frac{5}{3}.$
    \end{enumerate}
\end{claim}
\begin{proof}
Notice that $h_1 = 1$, $b_{1}^{(1)} = 1,$ $b_{c_1-1}^{(1)} = b_{2}^{(1)} = 2,$ and $b_{2}^{(1)}-b_{1}^{(1)} = 2^{-dp_0} = 1,$ so the statement is true for $m = 1$. Let $r \in [n-2] \setminus \{1\}$ and suppose the statement holds for $m = r-1$. We will show it holds for $m = r.$ 
         We have for each $i \in [c_{r}-1]\setminus \{1\}$ that $\phi^{(r)}_i$ and $\phi^{(r)}_{i-1}$ are identical on $\mathcal{Q}_{r}^d \setminus (\bigcup \widetilde{\mathcal{U}_{r}} \cap \Z^d),$ so
        \begin{equation} \label{bij differences}
        \begin{split}
            b_{i}^{(r)}-b_{i-1}^{(r)} &= \frac{1}{2^{dp_{r-1}}}\sum_{k=1}^{t_r}2^{dp_{r-2}}(b_{M^{(r,i)}_k}^{(r-1)}-b_{M^{(r,i-1)}_k}^{(r-1)})
            \leq \frac{1}{2^{dp_{r-1}}}\sum_{k=1}^{t_r} \left(M^{(r,i)}_{k} - M^{(r,{i-1})}_{k}\right)\prod_{l=1}^{r-1} h_l \leq \frac{1}{2^{dp_{r-1}}}\prod_{l=1}^{r} h_l.
            \end{split}
        \end{equation}
We used the induction hypothesis, \eqref{eq:Mincreasing} and the last part of \ref{eqn: Mn difference}. Therefore, point \eqref{bij prop:point 1} is true for $m = r.$ 

Next we prove point \eqref{bij prop:point 2} for $m=r.$ 
    Recall that $\mc{Q}_{r}^{d}$ may be decomposed as the union of $U\cap \Z^{d}$ with $U\in \mc{U}_{r}$. Moreover, we have $\widetilde{\mc{U}_{r}}=(\widetilde{\mc{U}_{r,k}})_{k\in [t_{r}]}\subseteq \mc{U}_{r}$.
    Therefore, we can decompose $b_{1}^{(r)}$ as 
\begin{equation*}
    b_{1}^{(r)} = \frac{1}{2^{d(p_{r-1}-p_{r-2})}} \left({2^{-dp_{r-2}}}\sum_{k\in [t_{n}]}\sum_{x \in \widetilde{\mc{U}_{r,k}}\cap \Z^d} \phi_1^{(r)}(x) + {2^{-dp_{r-2}}}\sum_{U\in \mc{U}_{r}\setminus \widetilde{U_{r}}}\sum_{x\in U\cap \Z^{d}} \phi_1^{(r)}(x) \right).
\end{equation*}
Applying \eqref{eq:tn}, \eqref{eq:construction}, \eqref{eq:first_and_last} with $n=r$, $1\leq \phi_{1}^{(r)}(x)\leq 2$ for all $x\in \mc{Q}_{r}^{d}$ and finally the induction hypothesis to bound $b_{1}^{(r-1)}$, we get
    \begin{equation*}
        \begin{split}
            b_{1}^{(r)} &\leq \frac{b_{1}^{(r-1)}(2^{d(p_{r-1}-p_{r-2})}-Dc_{r-1}^D)+ 2(D c_{r-1}^D)}{2^{d(p_{r-1}-p_{r-2})}} = b_{1}^{(r-1)} + \frac{Dc_{r-1}^D(2-b_{1}^{(r-1)})}{2^{d(p_{r-1}-p_{r-2})}} \leq b_{1}^{(r-1)} + c_{r-1}^D 2^{-d(p_{r-1}-p_{r-2}-1)} 
            \\& \leq 1+\sum_{i=1}^{r-2}c_i^D2^{-d(p_{i}-p_{i-1}-1)} + c_{r-1}^D2^{-d(p_{r-1}-p_{r-2}-1)} = 1+\sum_{i=1}^{r-1}c_i^D2^{-d(p_{i}-p_{i-1}-1)}.
            \end{split}
    \end{equation*}
 Similarly, using a similar decomposition to $b_1^{(r)}$ above we have 
       \begin{equation*}
        \begin{split}
            b_{c_r-1}^{(r)} &\geq \frac{b_{c_{r-1}-1}^{(r-1)}(2^{d(p_{r-1}-p_{r-2})}-Dc_{r-1}^D)+ Dc_{r-1}^D}{2^{d(p_{r-1}-p_{r-2})}} \geq b_{c_{r-1}-1}^{(r-1)} - c_{r-1}^D{2^{-d(p_{r-1}-p_{r-2}-1)}} \\
            & \geq 2-\sum_{i=1}^{r-2}c_i^D2^{-d(p_{i}-p_{i-1}-1)} - c_{r-1}^D2^{-d(p_{r-1}-p_{r-2}-1)} = 2-\sum_{i=1}^{r-1}c_i^D2^{-d(p_{i}-p_{i-1}-1)}.
            \end{split}
    \end{equation*}
    Finally, by the second condition of (\ref{cn tn pn}), we have that $\sum_{i=1}^{\infty}c_i^D2^{-d(p_{i}-p_{i-1}-1)}\leq \frac{1}{3},$
    and the claim follows. 
    \end{proof} 
    
By Claim \ref{bij prop} with $m = n-1$ and the fact that $\rho(x)\in\sqbr{\frac{4}{3},\frac{5}{3}}$ for all $x\in [0,1]^{d}$ it follows that for each $U \in \widetilde{\mathcal{U}_{n}}$ there exists $\alpha =\alpha_U \in [c_{n-1}-1]$ such that
        \begin{equation} \label{rho bk 1}
            \bigg|b_{\alpha}^{(n-1)}-{2^{d(p_{n-1}-p_{n-2})}}\int_{2^{-p_{n-1}}U}\rho \; d\mathcal{L}\bigg| \leq \frac{\prod_{l=1}^{n-1}h_l}{2^{dp_{n-2}}}.
        \end{equation}
          In case multiple $\alpha$ satisfy (\ref{rho bk 1}), choose $\alpha$ arbitrarily. Then we assign $\Z^d \cap U$ the colour $\alpha.$ That is, for each $x \in U\cap \Z^d$ we let
        \begin{equation} \label{eq: def of phicn}
            \phi^{(n)}_{c_n}(x)\coloneqq \phi^{(n-1)}_{\alpha}(x-\text{bp}(\Z^d \cap U)).
        \end{equation}
        Once again, notice that this construction obeys point $(\ref{palette 1})$ from Definition \ref{define: criteria for Psi}.
Using \eqref{eq: def of phicn} along with the fact that $\mathcal{Q}_{n-1}^d = \Z^d \cap U - \text{bp}(\Z^d \cap U)$ for each $U \in \widetilde{U_n},$
\begin{equation} \label{eq:balpha}
    2^{dp_{n-2}}b_{\alpha}^{(n-1)} \coloneqq \sum_{x \in \mathcal{Q}_{n-1}^d}\phi_{\alpha}^{(n-1)}(x) = \sum_{x \in \Z^d \cap U}\phi_{\alpha}^{(n-1)}(x-\text{bp}(\Z^d\cap U )) = \sum_{x \in \Z^d \cap U}\phi_{c_n}^{(n)}(x).
\end{equation}
This completes the construction of the good sequence of palettes $(\phi^{(n)}_{j})_{n\in\N,\,j\in[c_{n}]}$.
      Note $h_1 = 1$ and for each $l \geq 2$ that, from \eqref{cn hn} and \eqref{eq:tn},
        \begin{equation*} \label{eqn: hl bound}
            h_l = \bigg\lceil\frac{(c_{l-1}-2)t_l}{c_{l}-2}\bigg\rceil = \bigg\lceil\frac{(c_{l-1}-2)(2^{d(p_{l-1}-p_{l-2})}-Dc_{l-1}^D)}{c_{l}-2}\bigg\rceil \leq \frac{c_{l-1}-2}{c_l-2} 2^{d(p_{l-1}-p_{l-2}),}
        \end{equation*} 
        using the first condition of \eqref{cn tn pn} for the inequality. 
        Now, multiply \eqref{rho bk 1} through by $2^{dp_{n-2}}$, apply these bounds on $h_{l}$ and replace the first term with the final expression in \eqref{eq:balpha}. This delivers \eqref{rho bk} and completes the proof of Theorem \ref{intermediate thm}.
        \end{proof}

\begin{figure}[h]
\centering
\scalebox{1}{\begin{tikzpicture}
\centering
\draw[fill=red, opacity = 0.6,thick](12,-5) rectangle (13,-4);
\draw[red, fill=white] (12.5,-4.5) circle (0.5cm);
\draw[pink, outer color = pink, inner color = red] (12.5,-4.5) circle (0.5cm);
\node[draw=none] at (12,-5.3) {$0$};
\node[draw=none] at (13,-5.3) {$1$};
\node[draw=none] at (13,-6) {$\rho:[0,1]^2 \to[\frac{4}{3},\frac{5}{3}]$};
\draw[<->]        (10.3,-2.5)   -- (11.8,-4);
\draw[step=1cm,black,thick,yshift = 0.1cm] (0,0) grid (4,2);
\draw[step=1cm,black,thick,yshift = 0.1cm] (6,0) grid (10,2);
\draw[step=2cm,black,very thick,yshift = 0.1cm] (0,0) grid (4,2);
\draw[step=0.1cm,gray,very thin,yshift = 0.1cm] (0,0) grid (4,2);
\draw[step=0.1cm,gray,very thin,yshift = 0.1cm] (6,0) grid (10,2);
\node[draw=none] (ellipsis1) at (5,1) {$\hdots$};
\draw[step=2cm,black,very thick,yshift=0.1cm] (6,0) grid (10,2);
\node[yshift = -9cm] at (-0.5,4.5) {$\phi_1^{(n-1)}$};
\node[yshift = -9cm] at (1.5,4.5) {$\phi_{c_{n-1}}^{(n-1)}$};
\draw[->,yshift = -9cm]        (-0.5,4.3)   -- (0.2,4);
\draw[black, very thick] (0,0.1) -- (12,0.1);
\draw[black, very thick] (0,0.1) -- (0,3.1);
\draw[black, very thick] (10,0.1) -- (10,3.1);
\draw[->,yshift = -9cm]        (1.1,4.2)   -- (0.8,4);
\draw[<->] (0,-0.1) -- (10,-0.1);
\draw (5,-0.4) node[left]{$2^{p_n}$};
\draw[<->,yshift = -1cm]        (0,-0.2)   -- (1,-0.2);
\draw[->, very thick, blue] (-2,-6) -- (-2,3);
\draw[yshift = -9cm] (0,3.4) node[left]{$2^{p_{n-2}}$};
\draw[<->,yshift = -9cm]        (0.1,3.4)   -- (0.5,3.4);
\draw (0.5,-0.6) node[below]{$2^{p_{n-1}}$};
\draw [decorate,yshift = -9cm, decoration={brace,amplitude=5pt,mirror,raise=4ex}]
  (0.1,3.8) -- (0.9,3.8); 
\draw[yshift = -9cm] (1.3,3.7) 
node[right]{Palette of $c_{n-1}$ colours at level $n-1$};
\node[draw = none] at (-1,1){$\phi_1^{(n+1)}$};
\node[draw=none] (ellipsis1) at (10.5,2) {$\hdots$};
\draw[->,yshift = -9cm] (0.5,2.9) to [out=-50,in= -70] (1.5,3.4);
\fill[red, fill opacity=0.1, yshift = -5.5cm] (0.1,0) rectangle (0.5,0.4);
 rectangle (0.7,0.1);
\fill[red, fill opacity = 0.7, yshift = -5.5cm] (0.5,0) rectangle (0.9,0.4);
\draw[step=0.1cm,black,thin,yshift = -0.1cm] (0,0.2) grid (4,0.4);
\draw[step=0.1cm,black,thin,yshift = -0.1cm] (6,0.2) 
grid (10,0.4);
\draw[step=0.1cm,black,thin,yshift = 0.9cm] (0,0.2) grid (4,0.4);
\draw[step=0.2cm,black,thick,yshift = 0.9cm] (0,0.2) grid (4,0.4);
\draw[step=0.2cm,black,thick,yshift = -0.1cm] (0,0.2) grid (4,0.4);
\draw[step=0.2cm,black,thick,yshift = 0.9cm,xshift = 6cm] (0,0.2) grid (4,0.4);
\draw[step=0.2cm,black,thick,yshift = -0.1cm,xshift = 6cm] (0,0.2) grid (4,0.4);
\draw[step=0.1cm,black,thin,yshift = 0.9cm] (6,0.2) 
grid (10,0.4);
\fill[red, fill opacity = 0.6] (2,2) rectangle (2.1,2.1);
\fill[red, fill opacity = 0.7,yshift = 1cm] (7,0.3) rectangle (8,1.1);
\fill[red, fill opacity = 0.5,yshift = 0.1cm] (0.3,0.1) rectangle (0.6,0.2);
\fill[red, fill opacity = 0.5,yshift = 0.1cm] (0.7,0) rectangle (1,0.2);
\fill[red, fill opacity = 0.5,yshift = 0.1cm,xshift = 1cm] (0.3,0.1) rectangle (0.6,0.2);
\fill[red, fill opacity = 0.5,yshift = 0.1cm, xshift = 1cm] (0.7,0) rectangle (1,0.2);
\fill[red, fill opacity = 0.5,yshift = 0.1cm,xshift = 2cm] (0.3,0.1) rectangle (0.6,0.2);
\fill[red, fill opacity = 0.5,yshift = 0.1cm, xshift = 2cm] (0.7,0) rectangle (1,0.2);
\fill[red, fill opacity = 0.5,yshift = 0.1cm,xshift = 3cm] (0.3,0.1) rectangle (0.6,0.2);
\fill[red, fill opacity = 0.5,yshift = 0.1cm, xshift = 3cm] (0.7,0) rectangle (1,0.2);
\fill[red, fill opacity = 0.5,yshift = 1.1cm,xshift =0 cm] (0.3,0.1) rectangle (0.6,0.2);
\fill[red, fill opacity = 0.5,yshift = 1.1cm, xshift = 0cm] (0.7,0) rectangle (1,0.2);
\fill[red, fill opacity = 0.5,yshift = 1.1cm,xshift = 1cm] (0.3,0.1) rectangle (0.6,0.2);
\fill[red, fill opacity = 0.5,yshift = 1.1cm, xshift = 1cm] (0.7,0) rectangle (1,0.2);
\fill[red, fill opacity = 0.5,yshift = 1.1cm,xshift = 2cm] (0.3,0.1) rectangle (0.6,0.2);
\fill[red, fill opacity = 0.5,yshift = 1.1cm, xshift = 2cm] (0.7,0) rectangle (1,0.2);
\fill[red, fill opacity = 0.5,yshift = 1.1cm,xshift = 3cm] (0.3,0.1) rectangle (0.6,0.2);
\fill[red, fill opacity = 0.5,yshift = 1.1cm, xshift = 3cm] (0.7,0) rectangle (1,0.2);
\fill[red, fill opacity = 0.5,yshift = 0.1cm, xshift = 6cm] (0.3,0.1) rectangle (0.6,0.2);
\fill[red, fill opacity = 0.5,yshift = 0.1cm,xshift = 6cm] (0.7,0) rectangle (1,0.2);
\fill[red, fill opacity = 0.5,yshift = 0.1cm,xshift = 7cm] (0.3,0.1) rectangle (0.6,0.2);
\fill[red, fill opacity = 0.5,yshift = 0.1cm, xshift = 7cm] (0.7,0) rectangle (1,0.2);
\fill[red, fill opacity = 0.5,yshift = 0.1cm,xshift = 8cm] (0.3,0.1) rectangle (0.6,0.2);
\fill[red, fill opacity = 0.5,yshift = 0.1cm, xshift = 8cm] (0.7,0) rectangle (1,0.2);
\fill[red, fill opacity = 0.5,yshift = 0.1cm,xshift = 9cm] (0.3,0.1) rectangle (0.6,0.2);
\fill[red, fill opacity = 0.5,yshift = 0.1cm, xshift = 9cm] (0.7,0) rectangle (1,0.2);
\fill[red, fill opacity = 0.5,yshift = 1.1cm,xshift =6 cm] (0.3,0.1) rectangle (0.6,0.2);
\fill[red, fill opacity = 0.5,yshift = 1.1cm, xshift = 6cm] (0.7,0) rectangle (1,0.2);
\fill[red, fill opacity = 0.5,yshift = 1.1cm,xshift = 7cm] (0.3,0.1) rectangle (0.6,0.2);
\fill[red, fill opacity = 0.5,yshift = 1.1cm, xshift = 7cm] (0.7,0) rectangle (1,0.2);
\fill[red, fill opacity = 0.5,yshift = 1.1cm,xshift = 8cm] (0.3,0.1) rectangle (0.6,0.2);
\fill[red, fill opacity = 0.5,yshift = 1.1cm, xshift = 8cm] (0.7,0) rectangle (1,0.2);
\fill[red, fill opacity = 0.5,yshift = 1.1cm,xshift = 9cm] (0.3,0.1) rectangle (0.6,0.2);
\fill[red, fill opacity = 0.5,yshift = 1.1cm, xshift = 9cm] (0.7,0) rectangle (1,0.2);
\fill[red, fill opacity = 0.1,yshift = 0.1cm] (0,0) rectangle (4,2);
\fill[red, fill opacity = 0.7,yshift = 0.1cm] (6,1.8) rectangle (6.2,2);
\fill[red, fill opacity = 0.7,yshift = 0.1cm, xshift = 0.8cm] (6,1.8) rectangle (6.2,2);
\fill[red, fill opacity = 0.7,yshift = 0.1cm, xshift = 2cm] (6,1.8) rectangle (6.2,2);
\fill[red, fill opacity = 0.7,yshift = 0.1cm, xshift = 2.8cm] (6,1.8) rectangle (6.2,2);
\fill[red, fill opacity = 0.7,yshift = 0.1cm, xshift = 3cm] (6,1.8) rectangle (6.2,2);
\fill[red, fill opacity = 0.7,yshift = 0.1cm, xshift = 3.8cm] (6,1.8) rectangle (6.2,2);
\fill[red, fill opacity = 0.7,yshift = -0.9cm, xshift = 2cm] (6,1.8) rectangle (6.2,2);
\fill[red, fill opacity = 0.7,yshift = -1.5cm, xshift = 2cm] (6,1.8) rectangle (6.2,2);
\fill[red, fill opacity = 0.7,yshift = -0.9cm, xshift = 0cm] (6,1.8) rectangle (6.2,2);
\fill[red, fill opacity = 0.7,yshift = -1.5cm, xshift = 0cm] (6,1.8) rectangle (6.2,2);
\fill[red, fill opacity = 0.7,yshift = -0.9cm, xshift = 1cm] (6,1.8) rectangle (6.2,2);
\fill[red, fill opacity = 0.7,yshift = -1.5cm, xshift = 1cm] (6,1.8) rectangle (6.2,2);
\fill[red, fill opacity = 0.7,yshift = -0.9cm, xshift = 3cm] (6,1.8) rectangle (6.2,2);
\fill[red, fill opacity = 0.7,yshift = -1.5cm, xshift = 3cm] (6,1.8) rectangle (6.2,2);
\fill[red, fill opacity = 0.7,yshift = -0.9cm, xshift = 3.8cm] (6,1.8) rectangle (6.2,2);
\fill[red, fill opacity = 0.7,yshift = -1.5cm, xshift = 3.8cm] (6,1.8) rectangle (6.2,2);
\fill[red, fill opacity = 0.7,yshift = -0.9cm, xshift = 2.8cm] (6,1.8) rectangle (6.2,2);
\fill[red, fill opacity = 0.7,yshift = -1.5cm, xshift = 2.8cm] (6,1.8) rectangle (6.2,2);
\fill[red, fill opacity = 0.7,yshift = -0.9cm, xshift = 1.8cm] (6,1.8) rectangle (6.2,2);
\fill[red, fill opacity = 0.7,yshift = -0.5cm, xshift = 0cm] (6,1.8) rectangle (6.2,2);
\fill[red, fill opacity = 0.7,yshift = -1.5cm, xshift = 1.8cm] (6,1.8) rectangle (6.2,2);
\fill[red, fill opacity = 0.7,yshift = -0.9cm, xshift = 0.8cm] (6,1.8) rectangle (6.2,2);
\fill[red, fill opacity = 0.7,yshift = -1.5cm, xshift = 0.8cm] (6,1.8) rectangle (6.2,2);
\fill[red, fill opacity = 0.7,yshift = -0.9cm, xshift = 3.8cm] (6,1.8) rectangle (6.2,2);
\fill[red, fill opacity = 0.7,yshift = -1.5cm, xshift = 3.8cm] (6,1.8) rectangle (6.2,2);
\fill[red, fill opacity = 0.7,yshift = -0.5cm, xshift = 3.8cm] (6,1.8) rectangle (6.2,2);
\fill[red, fill opacity = 0.7,yshift = -0.5cm, xshift = 2.8cm] (6,1.8) rectangle (6.2,2);
\fill[red, fill opacity = 0.7,yshift = -0.5cm, xshift = 2cm] (6,1.8) rectangle (6.2,2);
\fill[red, fill opacity = 0.7,yshift = -0.5cm, xshift = 0.8cm] (6,1.8) rectangle (6.2,2);
\fill[red, fill opacity = 0.7,yshift = -1.2cm, xshift = 0.4cm] (6,1.8) rectangle (6.2,2);
\fill[red, fill opacity = 0.7,yshift = -1.2cm, xshift = 1.4cm] (6,1.8) rectangle (6.2,2);
\fill[red, fill opacity = 0.7,yshift = -1.2cm, xshift = 2.4cm] (6,1.8) rectangle (6.2,2);
\fill[red, fill opacity = 0.7,yshift = -1.2cm, xshift = 3.4cm] (6,1.8) rectangle (6.2,2);
\fill[red, fill opacity = 0.7,yshift = -0.2cm, xshift = 3.4cm] (6,1.8) rectangle (6.2,2);
\fill[red, fill opacity = 0.7,yshift = -0.2cm, xshift = 2.4cm] (6,1.8) rectangle (6.2,2);
\fill[red, fill opacity = 0.7,yshift = -0.2cm, xshift = 0.4cm] (6,1.8) rectangle (6.2,2);
\fill[red, fill opacity = 0.1] (6,0.1) rectangle (10,2.1);
\draw[step=0.1cm,gray,very thin,yshift = -0.2cm] (0,-2.21) grid (4,-1.2);
\draw[step=1cm,black,very thick,yshift = -0.4cm] (0,-2) grid (4,-1);
\draw[step=0.1cm,gray,very thin,yshift = -0.2cm,xshift = 6cm] (0,-2.21) grid (4,-1.2);
\draw[step=1cm,black,very thick,yshift = -0.4cm,xshift = 6cm] (0,-2) grid (4,-1);
\draw[step=0.2cm,black,very thick,yshift = -0.4cm,xshift = 0cm] (0,-2) grid (4,-1.8);
\draw[step=0.2cm,black,very thick,yshift = -0.4cm,xshift = 6cm] (0,-2) grid (4,-1.8);
\node[draw=none] (ellipsis1) at (5,-2) {$\hdots$};
\fill[red, fill opacity = 0.5,yshift = -2.4cm,xshift = 0cm] (0.7,0) rectangle (1,0.2);
\fill[red, fill opacity = 0.5,yshift = -2.4cm,xshift = 0cm] (0.3,0.1) rectangle (0.6,0.2);
\fill[red, fill opacity = 0.5,yshift = -2.4cm,xshift = 1cm] (0.7,0) rectangle (1,0.2);
\fill[red, fill opacity = 0.5,yshift = -2.4cm,xshift = 1cm] (0.3,0.1) rectangle (0.6,0.2);
\fill[red, fill opacity = 0.5,yshift = -2.4cm,xshift = 2cm] (0.7,0) rectangle (1,0.2);
\fill[red, fill opacity = 0.5,yshift = -2.4cm,xshift = 2cm] (0.3,0.1) rectangle (0.6,0.2);
\fill[red, fill opacity = 0.5,yshift = -2.4cm,xshift = 3cm] (0.7,0) rectangle (1,0.2);
\fill[red, fill opacity = 0.5,yshift = -2.4cm,xshift = 3cm] (0.3,0.1) rectangle (0.6,0.2);
\fill[red, fill opacity = 0.5,yshift = -2.4cm,xshift = 6cm] (0.7,0) rectangle (1,0.2);
\fill[red, fill opacity = 0.5,yshift = -2.4cm,xshift = 6cm] (0.3,0.1) rectangle (0.6,0.2);
\fill[red, fill opacity = 0.5,yshift = -2.4cm,xshift = 7cm] (0.7,0) rectangle (1,0.2);
\fill[red, fill opacity = 0.5,yshift = -2.4cm,xshift = 7cm] (0.3,0.1) rectangle (0.6,0.2);
\fill[red, fill opacity = 0.5,yshift = -2.4cm,xshift = 8cm] (0.7,0) rectangle (1,0.2);
\fill[red, fill opacity = 0.5,yshift = -2.4cm,xshift = 8cm] (0.3,0.1) rectangle (0.6,0.2);
\fill[red, fill opacity = 0.5,yshift = -2.4cm,xshift = 9cm] (0.7,0) rectangle (1,0.2);
\fill[red, fill opacity = 0.5,yshift = -2.4cm,xshift = 9cm] (0.3,0.1) rectangle (0.6,0.2);
\fill[red, fill opacity = 0.5,yshift = -1.5cm] (1,0) rectangle (1.1,0.1);
\fill[red, fill opacity = 0.5,yshift = -1.5cm] (2,0) rectangle (2.2,0.1);
\fill[red, fill opacity = 0.5,yshift = -1.5cm] (3,0) rectangle (3.3,0.1);
\fill[red,fill opacity = 0.7, yshift = -1.5cm,xshift = 7.9cm] (0.1,0) rectangle (1.1,0.1);
\fill[red,fill opacity = 0.7, yshift = -1.5cm,xshift = 5.9cm] (1.1,0) rectangle (2,0.1);
\fill[red,fill opacity = 0.7,yshift = -1.5cm,xshift = 3.9cm] (2.1,0) rectangle (2.9,0.1);
\fill[red, fill opacity = 0.7,yshift = -2.4cm,xshift = 6cm] (0,0.2) rectangle (3,0.9);
\fill[red, fill opacity = 0.1,yshift = -2.4cm,xshift = 0cm] (0,0) rectangle (4,1);
\fill[red, fill opacity = 0.1,yshift = -2.4cm,xshift = 9cm] (0,0) rectangle (1,1);
\fill[red, fill opacity = 0.1,yshift = -2.4cm,xshift =6cm] (0,0) rectangle (4,1);
\fill[red, fill opacity = 0.7,yshift = -2.4cm,xshift =9cm] (0,0.2) rectangle (0.2,0.4);
\fill[red, fill opacity = 0.7,yshift = -2.4cm,xshift =9.8cm] (0,0.2) rectangle (0.2,0.4);
\fill[red, fill opacity = 0.7,yshift = -1.8cm,xshift =9cm] (0,0.2) rectangle (0.2,0.4);
\fill[red, fill opacity = 0.7,yshift = -1.8cm,xshift =9.8cm] (0,0.2) rectangle (0.2,0.4);
\fill[red, fill opacity = 0.7,yshift = -2.1cm,xshift =9.4cm] (0,0.2) rectangle (0.2,0.4);
\node at (0.5,-3) {$\phi_1^{(n)}$};
\node at (1.5,-3) {$\phi_2^{(n)}$};
\node at (2.5,-3) {$\phi_3^{(n)}$};
\node at (3.5,-3) {$\phi_4^{(n)}$};
\node at (6.5,-3) {$\phi_{c_n-3}^{(n)}$};
\node at (7.5,-3) {$\phi_{c_n-2}^{(n)}$};
\node at (8.5,-3) {$\phi_{c_n-1}^{(n)}$};
\node at (9.5,-3) {$\phi_{c_n}^{(n)}$};
\draw [decorate,decoration={brace,amplitude=5pt,mirror,raise=4ex}]
  (0,-3) -- (10,-3) node[midway,yshift=-3em]{Palette of colours at level $n$};
\fill[red,fill opacity = 0.7, yshift = 1.3cm,xshift = 9cm] (0,0) rectangle (0.2,0.2);
\label{fig:construction}
\end{tikzpicture}}
\caption{How the palettes for $d=2$ are built at levels $n-1,$ $n$ and $n+1$ in Theorem \ref{intermediate thm} depicted with step size $h_n = 1$. If a small square is coloured pink, then its bottom-left corner is assigned the value $1$, and if it is coloured red, its bottom-left corner is assigned the value $2$. The palette at level $n$ contains colours with every $4-$combination of colours from the $(n-1)^{th}$ palette along its base. Moreover, the colours at level $n$ consist increasing shades from $\phi_1^{(n)}$ to $\phi_{c_{n}-1}^{(n)},$ while $\phi_{c_n}^{(n)}$ is an approximation of the density $\rho.$ If we increase $h_n,$ then the palette at level $n$ becomes `coarser'.}
\end{figure}
  
    \section{Non-rectifiability and repetitivity}
    In the present section we make use of the framework we have established in the previous two sections for constructing repetitive Delone sets. We verify that our construction provides repetitive Delone sets encoding any given density $\rho$ in the sense of Definition~\ref{defn: X encoding rho}.
    The next proposition is an extension of~\cite[Lemma~5.5]{dymond2018mapping}. It provides conditions under which a sequence of measures $(\nu_n)_{n\in\N}$ converges weakly to a measure $\nu$.
 \begin{prop}[\cite{dymond2018mapping}, Lemma $5.5,$ extended] \label{5.5 extended} Let $\nu$ and $(\nu_n)_{n \in \N}$ be finite Borel measures on a compact metric space $\mathcal{K}$. Assume that for each $n \in\N$ there are finite collections $\mathcal{T}_n\supseteq \widetilde{\mc{T}_{n}}$ of Borel subsets of $\mathcal{K}$ satisfying 
    \begin{equation} \label{prop conditions}
        \begin{split}
        \sum_{T \in \mathcal{T}_n}\nu(T) = \nu(\mathcal{K}),&\qquad \lim_{n \to \infty} \max\Bigg\{\nu_n\left(\mathcal{K}\setminus\bigcup_{T \in \widetilde{T}_n} T\right),\nu\left(\mathcal{K}\setminus\bigcup_{T \in \widetilde{T}_n} T\right)\Bigg\} = 0,\\
        \lim_{n \to \infty}\max_{T \in \mathcal{T}_n}\diam(T) = 0& \qquad \emph{and} \qquad \max_{T \in \widetilde{\mathcal{T}_n}}|\nu_n(T)-\nu(T)| \in o\left(\frac{1}{|\mathcal{T}_n|}\right).
        \end{split}
    \end{equation} Then $\nu_n \rightharpoonup \nu$.
    \end{prop}
    \begin{proof} We will first show 
    \begin{equation} \label{zeta bigger than 1}
    \int_\mathcal{K} \zeta \; d\nu_n \xrightarrow{n\rightarrow\infty} \int_\mathcal{K}\zeta \; d\nu \qquad \text{whenever } \zeta \in C(\mathcal{K},\R) \text{ with } \text{min }\zeta \geq 1.
\end{equation}
        Let $\epsilon\in (0,1)$ and $\zeta \in C(\mathcal{K},\R)$ satisfying $\text{min }\zeta\geq 1$. For each $T \in \widetilde{\mathcal{T}_n},$ choose $z_T \in T$ arbitrarily. Since $\mathcal{K}$ is compact, $\zeta$ is uniformly continuous on $\mathcal{K}$, and so by the third condition of (\ref{prop conditions}) there exists $N_1 \in \N$ such that
        \begin{equation*}
            |\zeta(z_T) - \zeta(x)| \leq \epsilon \qquad \text{for all $n \geq N_1,$ $T \in \widetilde{\mathcal{T}_n}$ and $x \in T$}.
        \end{equation*}
        By the last condition of (\ref{prop conditions}), there exists $N_2 \in \N$ such that
        \begin{equation*}
            \max_{T \in \widetilde{\mathcal{T}_n}}|\nu_n(T)- \nu(T)| \leq \frac{\epsilon}{|\mathcal{T}_n|}\qquad \text{for all $n \geq N_2$}
        \end{equation*}
   and by the second condition of \eqref{prop conditions} there exists $N_3 \in \N$ satisfying 
   \begin{equation*}
    \max\Bigg\{\nu_n\left(\mathcal{K}\setminus\bigcup_{T \in \widetilde{T}_n} T\right),\nu\left(\mathcal{K}\setminus\bigcup_{T \in \widetilde{T}_n} T\right)\Bigg\} \leq \epsilon \qquad \text{for each $n\geq N_{3}$}.
   \end{equation*}
    Now choose $N \coloneqq \max\{N_1,N_2,N_3\}$ and let $n\geq N$. Then for any $T\in\widetilde{\mc{T}_{n}}$ we have
    \begin{equation*}
        \begin{split}
            \int_T \zeta \; d\nu_n &\leq \int_T \zeta(z_T)+\epsilon \; d\nu_n = (\zeta(z_T)+\epsilon)\nu_n(T) \leq (\zeta(z_T)+\epsilon)\left(\nu(T)+\frac{\epsilon}{|\mathcal{T}_n|}\right)\\
            &\leq \int_T (\zeta+2\epsilon) \; d\nu + (\zeta(z_T)+\epsilon)\frac{\epsilon}{|\mathcal{T}_n|} \leq \int_T \zeta\; d\nu + 2\epsilon \nu(T) + (\zeta(z_T)+\epsilon)\frac{\epsilon}{|\mathcal{T}_n|}.
        \end{split}
    \end{equation*}
    Symmetrically, the lower bound obtained is
    \begin{equation*}
        \int_T \zeta\; d\nu_n \geq \int_T \zeta\; d\nu - 2\epsilon \nu(T) -(\zeta(z_T)-\epsilon)\frac{\epsilon}{|\mathcal{T}_n|}.
    \end{equation*}
   Summing up the above bounds over all $T \in \widetilde{\mathcal{T}_n}$ and using that $\zeta(z_T) - \epsilon >0$ and the first condition of (\ref{prop conditions}), we get
    \begin{equation*}
        \bigg|\int_{\bigcup_{T \in \widetilde{T}_n} T}\zeta \; d\nu_n - \int_{\bigcup_{T \in \widetilde{T}_n} T}\zeta \; d\nu\bigg| \leq 2\epsilon \nu(\mathcal{K}) + \epsilon(||\zeta||_\infty + \epsilon).
    \end{equation*} 
    Moreover,
    \begin{equation*}
    \begin{split}
        \bigg|\int_{\mathcal{K}\setminus \bigcup_{T \in \widetilde{T}_n} T}\zeta \; d\nu_n\bigg| \leq ||\zeta||_{\infty} \nu_n\left(\mathcal{K}\setminus\bigcup_{T \in \widetilde{T}_n} T\right) \leq ||\zeta||_\infty\epsilon,
        \end{split}
    \end{equation*} 
    with the same bound obtained for $\bigg|\int_{\mathcal{K}\setminus \bigcup_{T \in \widetilde{T}_n} T}\zeta \; d\nu\bigg|$. Therefore,
    \begin{equation*}
    \begin{split}
        \bigg|\int_\mathcal{K}\zeta \; d\nu_n - \int_\mathcal{K}\zeta \; d\nu\bigg| &\leq \bigg|\int_{\bigcup_{T \in \widetilde{T}_n} T}\zeta \; d\nu_n - \int_{\bigcup_{T \in \widetilde{T}_n} T}\zeta \; d\nu\bigg| + \bigg|\int_{\mathcal{K} \setminus {\bigcup_{T \in \widetilde{T}_n} T}}\zeta \; d\nu_n\bigg| + \bigg|\int_{\mathcal{K} \setminus {\bigcup_{T \in \widetilde{T}_n} T}}\zeta \; d\nu\bigg|\\
        &\leq 2\epsilon \nu(\mathcal{K}) + \epsilon(||\zeta||_\infty + \epsilon) + 2\epsilon||\zeta||_\infty.
        \end{split}
    \end{equation*} Since $\epsilon$ was arbitrary, we obtain \eqref{zeta bigger than 1}.

    Now suppose $\zeta \in C(\mathcal{K},\R),$ and we will show $\int_\mathcal{K} \zeta \; d\nu_n \xrightarrow{n\rightarrow\infty} \int_\mathcal{K}\zeta \; d\nu.$ If $\min\zeta \geq 1,$ then we are done by~\eqref{zeta bigger than 1}. Otherwise, if $\min \zeta < 1,$ we define the function $\psi \in C(\mathcal{K},\R)$ by $$\psi(x) = \zeta(x) + |\text{min } \zeta|+1.$$
    Notice that $\text{min }\psi = \text{min }\zeta + |\text{min }\zeta| +1 \geq 1,$ and $1+|\text{min }\zeta| \geq 1.$ Thus, using \eqref{zeta bigger than 1}, we have that
    $$\int_\mathcal{K}\zeta \; d\nu_n =\int_\mathcal{K} \psi \; d\nu_n -\int_\mathcal{K} (1+|\text{min } \zeta|) \; d\nu_n\xrightarrow{n\rightarrow\infty} \int_\mathcal{K}\psi \; d\nu-\int_\mathcal{K} (1+|\text{min } \zeta|) \; d\nu = \int_\mathcal{K} \zeta d\nu,$$ as required.
    \end{proof}
    We are now ready to prove our first main result:
   \mainlemma* 
   \begin{proof} 
   Let $D \coloneqq 2^d$. For each $n \in \N,$ let $\epsilon_n \coloneqq \frac{1}{(1-\frac{1}{d}) \log \log(n+100)} \in (0,2).$ Introduce the sequences $(p_n)_{n \in \N}$ and $(c_n)_{n \in \N}$ so that $p_0 \coloneqq 0,$ $c_1 \coloneqq 3$ and for each $n \in \N$
   \begin{equation*}
       p_n = p_{n-1} +1+\bigg\lfloor\frac{1}{d}\log_2(kn^{1+\epsilon_n}) \bigg\rfloor \qquad \text{and} \qquad c_n = \min\{3\lfloor n^{\frac{\epsilon_n}{Dd}}\rfloor,D(c_{n-1}-2)^{D+1}+2\},
   \end{equation*} where $k \coloneqq 3^{d(D+1)}\sum_{n=1}^\infty{n^{-(1+\epsilon_n(1-\frac{1}{d}))}}.$ It can be verified that $k$ is well-defined using the Cauchy condensation test along with for each $n \geq 16$ the inequality$$\frac{n}{\log\log(2^n+100)}\geq \frac{n}{\log\log(2^{n+1})}= \frac{n}{\log((n+1)\log2)}\geq \frac{n}{\log(n+1)}\geq \log(n^2),$$ and since the first term of the sum is $1,$ we see that $k \geq 3^{d(D+1)} \geq 9^d.$
    Let $(\phi_j^{(n)})_{j \in [c_n],n \in \N}$ be the sequence of palettes from Theorem \ref{intermediate thm} associated with $(p_n)_{n \in \N}, (c_n)_{n \in \N}$ and $\rho$.
    Let $\Psi:\Z^d \to [2]$ be the function from the conclusion of Theorem \ref{main construction} associated with $(\phi_j^{(n)})_{j \in [c_n],n \in \N}$. 
     Let $V_1 \coloneqq \{\frac{1}{2}\}^d$ and $V_2 \coloneqq \set{\frac{1}{4},\frac{3}{4}}\times\set{\frac{1}{2}}^{d-1}.$ Finally, let $X \coloneqq \bigcup_{z\in\Z^{d}}z+V_{\Psi(z)}$.
     
We claim the sequence of palettes, the function $\Psi$ and the set $X$ are well-defined since $(c_n)_{n \in \N}$ and $(p_n)_{n \in \N}$ satisfy the conditions (\ref{constraint cn pn}) and (\ref{cn tn pn}) respectively. Indeed, for each $n \in \N$, since $k \geq 9^d,$ we see that $p_n-p_{n-1} \geq 2.$ Next, we have for each $n \in \N$ that
\begin{equation*}
    2^{p_{n}-p_{n-1}-1} \geq 2^{\frac{1}{d}\log_2(kn^{1+\epsilon_n})-1} = \frac{1}{2}k^{\frac{1}{d}}n^{\frac{1+\epsilon_n}{d}} \geq 3^Dn^{\frac{\epsilon_n}{d}}\geq c_n^D,
\end{equation*} where the penultimate inequality holds since $k \geq 3^{d(D+1)}.$ By definition, $c_n \leq D(c_{n-1}-2)^{D+1}+2.$ Finally,
\begin{equation*}
    \sum_{n=1}^\infty c_n^D2^{-d(p_{n}-p_{n-1}-1)} \leq \sum_{n=1}^\infty 3^Dn^{\frac{\epsilon_n}{d}}2^{-d(\frac{1}{d}\log_2(kn^{1+\epsilon_n})-1)} =\frac{D3^D}{k}\sum_{n=1}^\infty n^{-1-\epsilon_n+\frac{\epsilon_n}{d}}=\frac{2^d3^D}{3^{d(D+1)}}\leq\frac{1}{3}.
\end{equation*}Observe for each $n\in \N$ that $c_n\geq 3$. Indeed, $c_1 \coloneqq 3$, and assuming $c_{n-1}\geq 3$, we have that either $c_n = 3\lfloor n^\frac{\epsilon_n}{Dd}\rfloor \geq 3,$ or $c_n = D(c_{n-1}-2)^{D+1}+2,$ in which case we see
   \begin{equation*}
         \frac{D(c_{n-1}-2)^{D+1}}{c_{n-1}-2} = D(c_{n-1}-2)^D \geq {D} >1,
     \end{equation*} meaning $c_n > c_{n-1} \geq 3.$ As a corollary, noting that $n^{\epsilon_n} \to \infty$ as $n \to \infty,$ we have that $c_n$ is unbounded. Thus, the conditions of \eqref{constraint cn pn} and \eqref{cn tn pn} are satisfied, and so $\Psi$ is well-defined. By point \eqref{eq: repetitivity function} of Theorem \ref{main construction}, $\Psi$ is a repetitive mapping, and so by Lemma \ref{rep from mapping to net}, $X$ is a repetitive Delone set.
     
Now for each $n \in \N$, let $\varphi_n$ be a bijective affine map taking $[0,1]^d$ to a $d-$dimensional cube of sidelength $2^{p_{n-1}}$ containing a cubic set $S_n \subset \Z^d$ of sidelength $2^{p_{n-1}}$ that satisfies for each $x \in S_n$
    \begin{equation*}
        \Psi(x) = \phi_{c_n}^{(n)}(x-\text{bp}(S_n)).
    \end{equation*} Note these bijective affine maps exist by point (\ref{condition: every colour found somewhere}) of Theorem \ref{main construction}. Let $\mu_n$ be the measure defined on $[0,1]^d$ so that 
    for each $A \subseteq [0,1]^d$
    \begin{equation*}
        \mu_n(A) \coloneqq \frac{1}{2^{dp_{n-1}}}|\varphi_n(A) \cap X|.
    \end{equation*}
        Now we use Proposition \ref{5.5 extended} with $\mathcal{K} = [0,1]^d,$ $\nu_n = \mu_n,$ $\nu = \rho\mathcal{L}$, and $\mathcal{T}_n =2^{-p_{n-1}}\mathcal{U}_n,$ where we recall from the statement of Theorem \ref{intermediate thm} that
        \begin{equation*}
            \mathcal{U}_n \coloneqq \bigg\{\prod_{i=1}^{d}[(m_{i}-1)2^{p_{n-2}},m_{i}2^{p_{n-2}}):(m_{1},\ldots,m_{d})\in [2^{p_{n-1}-p_{n-2}}]^d\bigg\},
        \end{equation*} and $\widetilde{\mathcal{U}_n} \subset \mathcal{U}_n$ with $|\widetilde{\mathcal{U}_n}| = 2^{d(p_{n-1}-p_{n-2})}-Dc_{n-1}^D.$ We set $\widetilde{\mathcal{T}_n} = 2^{-p_{n-1}}\widetilde{\mathcal{U}_n} \subset \mathcal{T}_n$. We must verify that the conditions of Proposition \ref{5.5 extended} hold.  For each $n \geq 2,$ since $c_{n-1}^D \leq 2^{p_{n-1}-p_{n-2}-1}$
        \begin{equation*}
            \mathcal{L}\left(\mathcal{K}\setminus\bigcup_{T \in \widetilde{T}_n} T\right)  = \frac{Dc_{n-1}^D}{2^{d(p_{n-1}-p_{n-2})}} \leq \frac{D2^{p_{n-1}-p_{n-2}-1}}{2^{d(p_{n-1}-p_{n-2})}} = \frac{1}{2^{(d-1)(p_{n-1}-p_{n-2}-1)}} \xrightarrow{n\rightarrow\infty} 0.
        \end{equation*}
     Since $\rho(x) \leq 2$ for all $x \in \mathcal{K}$
\begin{equation*}\nu\left(\mathcal{K}\setminus\bigcup_{T \in \widetilde{T}_n} T\right) = \int_{\mathcal{K}\setminus\bigcup_{T \in \widetilde{T}_n} T} \rho \; d\mathcal{L} \leq 2\mathcal{L}\left(\mathcal{K}\setminus\bigcup_{T \in \widetilde{T}_n} T\right) \xrightarrow{n\rightarrow\infty} 0.
    \end{equation*} Moreover, since $X$ contains at most two points in each unit volume cube in $\R^d$ with integer coordinates
    \begin{equation*} \nu_n\left(\mathcal{K}\setminus\bigcup_{T \in \widetilde{T}_n} T\right) = \frac{1}{2^{dp_{n-1}}}\Bigg|\varphi_n\left(\mathcal{K}\setminus\bigcup_{T \in \widetilde{T}_n} T\right) \cap X\Bigg| \leq \frac{1}{2^{dp_{n-1}}}2\mathcal{L}\left(\varphi_n\left(\mathcal{K}\setminus\bigcup_{T \in \widetilde{T}_n} T\right)\right) = 2\mathcal{L}\left(\mathcal{K}\setminus\bigcup_{T \in \widetilde{T}_n} T\right) \xrightarrow{n\rightarrow\infty} 0.
    \end{equation*}The first two conditions of (\ref{prop conditions}) are clear. Thus, it suffices to show that
        \begin{equation*}
            \max_{T \in \widetilde{\mathcal{T}_n}}|\mu_n(T) - \rho \mathcal{L}(T)| \in o\left(\frac{1}{|\mathcal{T}_n|}\right).
        \end{equation*}
       Note that $|\mathcal{T}_n| = |\mathcal{U}_n| = 2^{d(p_{n-1}-p_{n-2})}.$ For each $T = 2^{-p_{n-1}}U\in \widetilde{\mathcal{T}_n}$
        \begin{equation*}
            \begin{split}
                \mu_n(T) = \frac{1}{2^{dp_{n-1}}}|\varphi_n(T) \cap X|= \frac{1}{2^{dp_{n-1}}}\sum_{x \in \varphi_n(T) \cap \Z^d}\Psi(x)&=\frac{1}{2^{dp_{n-1}}}\sum_{x \in \varphi_n(T) \cap \Z^d}\phi_{c_n}^{(n)}(x-\text{bp}(S_n)) \\ &=\frac{1}{2^{dp_{n-1}}}\sum_{x \in U \cap \Z^d}\phi_{c_n}^{(n)}(x).
            \end{split}
        \end{equation*}
        Hence, from \eqref{rho bk}, we get
        \begin{equation*}
            \abs{\mu_{n}(T)-\rho\leb(T)}\leq \frac{1}{\br{c_{n-1}-2}2^{d(p_{n-1}-p_{n-2})}}=\frac{1}{(c_{n-1}-2)\abs{\mc{T}_{n}}}.
        \end{equation*}
Therefore, $\mu_n \rightharpoonup \rho\mathcal{L},$ and so $X$ encodes $\rho.$

Let $\widetilde{R}(r)$ stand for the repetitivity function of $\Psi$ from Theorem~\ref{main construction}\eqref{eq: repetitivity function}. By Lemma~\ref{rep from mapping to net}, we have that $X$ has repetitivity function belonging to $O(\widetilde{R}(r))$. To complete the proof, we now show that $\widetilde{R}(r)\in \bigcap_{q>2/d}O(r(\frac{\log r}{\log \log r})^q)$. Fix $q>2/d$. Let $N\geq2$ be large enough so that $p_{N-2}>3$ and $\frac{2(1+\epsilon_{n-1})}{d}<q$ for all $n\geq N$. Now suppose $r>2^{p_{N-2}-1}$ and let $n\geq N$ be the unique integer satisfying $2^{p_{n-2}}<2r\leq 2^{p_{n-1}}$. Notice that 
    \begin{equation}\label{eqn: pnpn2}
        \begin{split}
        2^{p_n-p_{n-2}}= 2^{2+(p_{n}-p_{n-1}-1)+(p_{n-1}-p_{n-2}-1)} \leq4\;2^{\frac{1}{d}(\log_2(kn^{1+\epsilon_n})+\log_2(k(n-1)^{1+\epsilon_{n-1}}}))\\
        \leq 4k^\frac{2}{d}n^\frac{2(1+\epsilon_{n-1})}{d}\leq 4k^{\frac{2}{d}}n^{q}.
        \end{split}
    \end{equation}
   Since $\sum_{j=1}^{n-2} \log_2(j) \geq \int_1^{n-2}\log_2(x)\; dx$ and $\log_2(k) \geq \log_2(9^d) \geq \log_2(64) = 6$, we have
    \begin{equation*}
    \begin{split}
        p_{n-2}&= n-2+ \sum_{j=1}^{n-2}\bigg\lfloor\frac{1}{d}\log_2(kj^{1+\epsilon_j})\bigg\rfloor \geq\frac{1}{d} \sum_{j=1}^{n-2}\log_2(kj)\geq\frac{1}{d}\br{(\log_2k)(n-2)+\sum_{j=1}^{n-2}\log_2(j)}\\
        &\geq \frac{1}{d}\br{6(n-2) + \frac{(n-2)(\log(n-2)-1)+1}{\log 2}} \geq \frac{(n-2)(5+\log(n-2))+1}{d}\geq \frac{n\log_2 n +1}{d}.
        \end{split}
    \end{equation*}
    Therefore,
    \begin{equation*}
        n^n=2^{n \log_2 n} \leq 2^{dp_{n-2}-1}<2^{d-1}r^{d},
    \end{equation*}
     so using that $\frac{\log x}{\log(x\log x)}\geq \frac{1}{2}$ for all $x>1$ and that the function $\frac{\log x}{\log\log x}$ is increasing for all $x \geq x_0$ for a suitably large $x_0>e$, we get
    \begin{equation}\label{eqn: pn n}
        \frac{1}{2}n \leq  \frac{\log n^n}{\log\log n^n} < \frac{\log(2^{d-1}r^{d})}{\log\log(2^{d-1}r^{d})}\leq \frac{2d\log r}{\log\log r}.
    \end{equation}
    Then by using (\ref{eqn: pnpn2}) and (\ref{eqn: pn n})
    \begin{equation*}
            \frac{\widetilde{R}(r)}{2r} \leq \sqrt{d}\; 2^{p_{n}-p_{n-2}}\leq \sqrt{d}\; 4K^\frac{2}{d}(4d)^{q}\left(\frac{\log r}{\log \log r}\right)^{q}.
    \end{equation*}
    Hence, $\widetilde{R}(r)\in O\br{r\br{\frac{\log r}{\log\log r}}^{q}}$. 
    \end{proof} 
To end this section, we will give a formal proof for Theorem \ref{prop:non-rectifiable non-realisable}. As a reminder to the reader, this result is combined with Theorem \ref{delone}, together with the non-realisable densities given by Burago and Kleiner~\cite{burago1998separated}, to deliver our main result, Theorem \ref{1}. This gives the first examples of repetitive, non-rectifiable Delone sets with explicit, and moreover close to optimal, bounds on the repetitivity function.

 \realisable*
    \begin{proof} 
        This statement is attributed to Burago and Kleiner \cite[Lemma $2.1$]{burago1998separated}, with a complete proof contained inside the proof of Lemma~3.4 in \cite{dymond2023highly}, where the following more general statement is shown: If $X$ encodes $\rho$ which does not admit any bi-$\omega$ continuous solution $f\colon [0,1]^{d}\to \R^{d}$ of the equation $f\sharp \rho\leb=\leb|_{f([0,1]^{d})}$, then there is no homogeneous bi-$\omega$ bijection $X\to \Z^{d}$. The reader of the present paper does not need to know the meaning of the notions involving $\omega$, only that $\omega$ stands for a modulus of continuity and that when we take $\omega(t)=t$ the above statement becomes precisely the statement of the present lemma. Below we include a proof of Theorem~\ref{prop:non-rectifiable non-realisable} in succinct form, omitting one delicate part of the argument, for which we refer the reader to \cite[Proof of Lemma~3.4]{dymond2023highly}.
        
        For each $n \in \N$, since $X$ encodes $\rho$, define $r_n,$ $Q_n,$ $\varphi_n$ and $\mu_n$ congruous with Definition \ref{defn: X encoding rho}.
    Suppose that $X$ is rectifiable for a contradiction, that is, there exist $L\geq 1$ and an $L$-bi-Lipschitz bijection $f:X \rightarrow \Z^d.$ The proof will be structured as follows. We will show there exists a bi-Lipschitz mapping $g:[0,1]^d \to \R^d$ such that 
        $g \# (\rho \mathcal{L}) = \mathcal{L}|_{g([0,1]^d)}.$
    Then by Rademacher's theorem, $\text{Jac}(g)$ exists $a.e.$ on $[0,1]^d,$ and thus by changing variables, we must have that $\rho = \text{Jac}(g) \; a.e.,$ which is a contradiction to the non-realisability of $\rho.$  
    For each $n \in \N$, let $\star_n$ be an arbitrary point in $\varphi_n^{-1}(X) \cap [0,1]^d,$ and define the sequence of bijections $g_n:\varphi_n^{-1}(X) \to  \frac{1}{r_n}\Z^d$ to be
    \begin{equation*}
        g_n(x) \coloneqq \frac{1}{r_n} (f\circ \varphi_n(x)- f\circ \varphi_n(\star_n)).
    \end{equation*}
    Notice that each $g_n$ is $L$-bi-Lipschitz. Indeed, for $x,y \in \varphi_{n}^{-1}(X)$
    \begin{equation*}
        \begin{split}
            ||g_n(x)-g_n(y)||_2 = \frac{1}{r_n}||f \circ \varphi_n(x) - f \circ \varphi_n (y)||_2
             \leq \frac{L}{r_n}\text{Lip}(\varphi_n)||x-y||_2 = L||x-y||_2,
        \end{split}
    \end{equation*} with a similar calculation for $g^{-1}_n.$ By the Kirszbraun extension theorem \cite{kirszbraun1934zusammenziehende}, for each $n \in \N,$ $g_n|_{[0,1]^d}$ can be extended to an $L-$Lipschitz mapping $\overline{g_n}:[0,1]^d \to \R^d.$ Moreover, $(\overline{g_n})_{n \in \N}$ is bounded. Indeed, given $n \in \N$ and $x \in [0,1]^d,$ there exists a point $y \in \varphi_n^{-1}(X)$ such that $||x-y||_2 \leq S,$ where $S = S(X)$ is the constant of relative density of $X$. Therefore, 
    \begin{equation*}
    \begin{split}
        ||\overline{g_n}(x)||_2 &\leq ||\overline{g_n}(x)-\overline{g_n}(y)||_2 + ||\overline{g_n}(y)||_2 \leq L||x-y||_2 + \frac{1}{r_n}||f\circ \varphi_n(y) - f\circ\varphi_n(\star_n)||_2\\
        &\leq LS + \frac{L}{r_n}\text{Lip}(\varphi_n)||y-\star_n||_2 = LS+L||y-\star_n||_2 \leq LS + L\sqrt{d}. 
        \end{split}
    \end{equation*}
    Thus, by the Arzelà-Ascoli theorem, the sequence uniformly subconverges to an $L$-Lipschitz mapping $g:[0,1]^d \to \R^d$. Using that each $\overline{g_{n}}$ is actually $L$-bilipschitz on $\varphi_{n}^{-1}(X)$, a $\frac{\sqrt{d}}{r_{n}}$-net of $[0,1]^{d}$, it is straightforward to show that their uniform limit (along a subsequence) $g\colon [0,1]^{d}\to \R^{d}$ is $L$-bilipschitz. Relabel $\overline{g_n}$ as necessary to include only elements from a subsequence that uniformly converges to $g$. By the assumption that $X$ encodes $\rho$, we know that $\mu_n\rightharpoonup \rho\mathcal{L}$. Then by \cite[Lemma~5.6]{dymond2018mapping}, since $\overline{g_n}$ uniformly converges to $g,$ we get that ${\overline{g_n}} \# \mu_n \rightharpoonup g \# (\rho\mathcal{L})$. Separately, by the delicate argument of \cite[Proof of Lemma~3.4]{dymond2018mapping} read with $\omega(t)=t$, we have that ${\overline{g_n}} \# \mu_n \rightharpoonup \mathcal{L}|_{g([0,1]^d)}.$ Thus, by the uniqueness of weak limits, we get that $g\#(\rho\mathcal{L}) = \mathcal{L}|_{g([0,1]^d)}$. 
    \end{proof}

    \section{Optimality of the repetitivity function}
   The repetitive Delone sets constructed in the present paper are given by Theorem~\ref{main construction}, corresponding to sequences $(p_{n})_{n\in\N}$ and $(c_{n})_{n\in\N}$ of positive integers and a sequence $(\phi_{j}\upp{n})_{n\in\N,\,j\in[c_{n}]}$ of good palettes. In order to encode a given density $\rho$ into a repetitive Delone set, we apply Theorem~\ref{main construction} to a special sequence of good palettes, tailored to $\rho$, given by Theorem~\ref{intermediate thm}. A condition for this tailoring in Theorem~\ref{intermediate thm}, is that $\sum_{n=1}^\infty 2^{-d(p_n-p_{n-1})}<\infty$. The sequence $p_{n}$ determines the repetitivity function of the Delone sets we construct: the construction ensures, roughly speaking, that every pattern of scale $2^{p_{n-1}}$ is found inside every cube of sidelength $2^{p_{n}}$. Thus, the condition $\sum_{n=1}^\infty 2^{-d(p_n-p_{n-1})}<\infty$ provides the main constraint on the repetitivity function we may achieve. 
    
    In the present section, we show that the condition $\sum_{n=1}^\infty 2^{-d(p_n-p_{n-1})}<\infty$ cannot be relaxed, if we are to construct repetitive Delone sets encoding a given density $\rho$, verifying Theorem~\ref{delone}.
     \begin{define}
        Let $X$ be a Delone set, $x \in X$ and $r>0.$ An \emph{($r-$)cubic patch} centred at $x$, denoted as $\mathcal{P}^{(Q)}_{x,r}$, is 
    \begin{equation*}
        \mathcal{P}^{(Q)}_{x,r} \coloneqq X \cap Q(x,r)
    \end{equation*} where $Q(x,r)$ is the open cube centred at $x$ with sidelength $r$ and sides parallel to the axes. An \emph{$X-$translate} of $\mathcal{P}^{(Q)}_{x,r}$ is an $r-$cubic patch $\mathcal{P}^{(Q)}_{y,r}$ with $y \in X$ and $$\mathcal{P}^{(Q)}_{y,r} = y-x+\mathcal{P}^{(Q)}_{x,r}.$$ Occasionally, when $r$ is specified, we abbreviate the notation $\mathcal{P}^{(Q)}_{x,r}$ to $\mathcal{P}^{(Q)}_{x}$.
    \end{define}
    \begin{thm} \label{thm: pn rep}
    Let $d \in \N$, $L \geq l>0$, $\rho:[0,1]^d \to \R_{>0}$ be a density, and let $X$ be a repetitive Delone set in $\R^d$ with covering radius $L$ and packing radius $l$ that encodes $\rho$. Let $p_0 \coloneqq \log_2(4L(\frac{16dL}{l})^d)$ and define the sequence $(p_n)_{n \in \N}$ by
     \begin{equation} \label{eq: defining sequence pn}
     \begin{split}
          p_n \coloneqq \min& \{m \in \N: \\ &\emph{each open cube of sidelength $2^m$ in $\R^d$ contains an $X-$translate of every  $2^{p_{n-1}}-$cubic patch of $X$}\}.
     \end{split}
 \end{equation}
   If $\sum_{n=1}^\infty 2^{-d(p_n-p_{n-1})} = \infty,$ then $\rho$ is constant almost everywhere.
\end{thm}
\begin{remark}
    Note the sequence $(p_n)_{n \in \N}$ in Theorem~\ref{thm: pn rep} is subtly different to the sequence $(p_n)_{n \in \N}$ from Theorems~\ref{main construction} and \ref{intermediate thm}. The first is determined from a given Delone set, and the second is chosen before constructing a Delone set. However, both sequences represent the same concept, namely, that every cube of sidelength $2^{p_n}$ contains an $X-$translate of every $2^{p_{n-1}}-$cubic patch of $X.$    
 \end{remark}

Before we prove Theorem \ref{thm: pn rep}, we require three lemmas. 

    \begin{lemma} \label{claim: K}
       Let $d \in \N,$ $l>0$ and $X \subset \R^d$ be a Delone set with packing radius $l$. Then every closed cuboid $Q \subset \R^d$ with sidelengths all at least $l$ contains at most $K\leb(Q)$ points of $X,$ where $K \coloneqq (\frac{2\sqrt{d}}{l})^d.$
    \end{lemma}
    \begin{proof}
          Every closed cube of sidelength $\frac{l}{\sqrt{d}}$ in $\R^d$ is contained within an open ball of radius $l,$ so it contains at most one point of $X.$ Now let $Q$ be a closed cuboid in $\R^d$ of sidelengths $s_1,s_2,...,s_d$ all at least $l$. Then it can be covered with $\prod_{i=1}^d\left\lceil \frac{s_i\sqrt{d}}{l}\right\rceil$ closed cubes of sidelength $\frac{l}{\sqrt{d}}$, each containing at most one point of $X.$ Therefore, since $s_i \geq l$ for each $i \in [d]$, the total number of points in $Q \cap X$ is
         \begin{equation*}
            |Q \cap X| \leq \prod_{i=1}^d\bigg\lceil \frac{s_i\sqrt{d}}{l}\bigg\rceil \leq  \prod_{i=1}^d\left(\frac{2s_i\sqrt{d}}{l}\right) = \left(\frac{2\sqrt{d}}{l}\right)^{d} \prod_{i=1}^ds_i = \left(\frac{2\sqrt{d}}{l}\right)^d \leb(Q)= K\leb(Q).
         \end{equation*}
    \end{proof}
    As in \cite{lagarias2002local}, we will use the term \emph{ideal crystal} to describe a Delone set $X\subset \R^d$ of the form $X = \Lambda +F$ with $\Lambda$ a full-rank lattice and $F$ a finite set in $\R^d.$ The next lemma asserts that these very special Delone sets arise in Theorem~\ref{thm: pn rep} in the case where $(p_n)_{n \in \N}$ is eventually constant. It is here where the choice of $p_{0}$ in Theorem~\ref{thm: pn rep} matters. We require that the cubic patches considered in the definition of $p_n$ are sufficiently large so that we can use a fact about the patch counting function, \cite[Theorem $2.1$]{lagarias2002local}, to show that $X$ is an ideal crystal.
\begin{lemma} \label{lemma: pn constant X ideal}
    Let $d \in \N,$ $L \geq l>0$ and $X$ be a repetitive Delone set in $\R^d$ which has covering radius $L$ and packing radius $l$. Let $p_0 \coloneqq \log_2(4L(\frac{16dL}{l})^d)$ and define the sequence $(p_n)_{n \in \N}$ as in \eqref{eq: defining sequence pn}. If $(p_n)_{n \in \N}$ is eventually constant, then $X$ is an ideal crystal.
\end{lemma}
\begin{proof}The sequence $(p_n)_{n \in \N}$ is well-defined since $X$ is repetitive. Let $M \in \N$ be such that $(p_n)_{n \geq M}$ is constant.
Let $x, y \in X$ and label the $2^{p_M}-$cubic patches centred at $x$ and $y$ respectively as $\mathcal{P}^{(Q)}_x=\mathcal{P}^{(Q)}_{x,2^{p_M}}$ and $\mathcal{P}^{(Q)}_y=\mathcal{P}^{(Q)}_{y,2^{p_{M}}}$. Then by the definition of $2^{p_{M+1}} = 2^{p_M}$, the open cube $Q(x,2^{p_{M}})$ contains an $X-$translate of $\mathcal{P}^{(Q)}_y$, and the open cube $Q(y,2^{p_{M}})$ contains an $X-$translate of $\mathcal{P}^{(Q)}_x$. This means there exist $z,v \in X$ such that, using similar abbreviated notation $\mathcal{P}^{(Q)}_{z}$ and $\mathcal{P}^{(Q)}_{v}$ for the $2^{p_{M}}$-cubic patches centred at $z$ and $v$ respectively,
\begin{equation} \label{eq: z}
    \mathcal{P}_y^{(Q)} = y-z+ \mathcal{P}_z^{(Q)},\qquad \mathcal{P}_{x}^{(Q)}=x-v+\mathcal{P}_{v}^{(Q)},
\end{equation}
\begin{equation} \label{eq: patch in patch}
    z-y+\mathcal{P}_y^{(Q)} \subseteq Q(x,2^{p_M}) \cap X =\mathcal{P}_x^{(Q)} \qquad \text{and} \qquad v-x+\mathcal{P}_x^{(Q)} \subseteq Q(y,2^{p_M}) \cap X =\mathcal{P}_y^{(Q)}.
\end{equation}
Taking the cardinality of each side of the two set inclusions in \eqref{eq: patch in patch}, we obtain that 
    $|\mathcal{P}^{(Q)}_x| = |\mathcal{P}^{(Q)}_y|.$
Thus, the set inclusions in \eqref{eq: patch in patch} are in fact equalities. Combining \eqref{eq: z} with the first equality of \eqref{eq: patch in patch}, we get that $\mathcal{P}_{z}^{(Q)} = \mathcal{P}_{x}^{(Q)}.$ Thus, we have that $$X \cap (Q(z,2^{p_M})\setminus Q(x,2^{p_{M}})) = \varnothing.$$ Hence, $Q(z,2^{p_M})\setminus Q(x,2^{p_{M}})$, where $2^{p_M} \geq 2^{p_0}\geq 4L$, does not contain an open ball with radius $L$. This implies $||x-z||_\infty < 2L.$  

Since $y$ was chosen arbitrarily, we have shown that for each $y \in X,$ there exists some $z \in X \cap Q(x,4L)$ such that $\mathcal{P}_{z,2^{p_{M}}}^{(Q)}$ is an $X-$translate of $\mathcal{P}_{y,2^{p_{M}}}^{(Q)}.$ Therefore, there are at most $\abs{X\cap Q(x,4L)}$ $2^{p_{M}}$-cubic patches of $X$ that pairwise are not $X$-translates of each other. In order to apply a result of \cite{lagarias2003repetitive} for the patch counting function, we need to express this in terms of patches, in the sense of Definition~\ref{def: rep_function}, based on the Euclidean metric, rather than cubic patches. Each $2^{p_M-1}-$patch (as in Definition \ref{def: rep_function}) of $X$ is contained in a $2^{p_M}-$cubic patch, so there are at most $|X \cap Q(x,4L)|$ ${2^{p_M-1}}-$patches that pairwise are not $X-$translates. In other words, the patch-counting function $N_X$ (see \cite[Definition $2.3$]{lagarias2002local}) evaluated at $2^{p_{M}-1}$ gives at most $|X \cap Q(x,4L)|$. By Lemma \ref{claim: K} along with the fact that $(p_n)_{n \in \N}$ is non-decreasing, we have that
$$N_X(2^{p_{M}-1}) \leq |X \cap Q(x,4L)| \leq \left(\frac{2\sqrt{d}}{l}\right)^d(4L)^d < \left(\frac{16dL}{l}\right)^d=\frac{2^{p_0-1}}{2L} \leq \frac{2^{p_M-1}}{2L}.$$
Applying \cite[Theorem $2.1$]{lagarias2002local} with $T= {2^{p_M-1}},$ we conclude that $X$ is an ideal crystal. 
    
\end{proof}
\begin{lemma} \label{lemma: X ideal rho constant}
    Let $d \in \N,$ $X \subset \R^d$ be an ideal crystal and $\rho:[0,1]^d \to \R_{>0}$ be a density. If $X$ encodes $\rho,$ then $\rho$ is constant almost everywhere.
\end{lemma}
\begin{proof}
  Since $X$ encodes $\rho,$ there exist sequences $(r_n)_{n \in \N} \subset \R_{>0}$ going to $\infty$, closed cubes $(Q_n)_{n \in \N} \subset \R^d$ of sidelength $r_n,$ bijective affine mappings $\varphi_n:[0,1]^d \to Q_n$, and measures $(\mu_n)_{n \in \N}$ defined by 
    \begin{equation*}
         \mu_n(A) = \frac{1}{r_n^d}|X \cap \varphi_n(A)|, \qquad A \subseteq [0,1]^d
    \end{equation*} satisfying that $\mu_n \rightharpoonup \rho\mathcal{L}.$
    Let $\Lambda$ be the lattice with basis $(\lambda_1,...,\lambda_d)$ inducing $X,$ and let $G\subset \R^d$ be given by 
    $$G = \{t_1\lambda_1+...+t_d\lambda_d:t_i \in [0,1) \text{ for each } i \in [d]\}.$$
    It is quick to check that $G$ is a bounded set that satisfies $|G \cap X| >0,$ $\bigsqcup_{\lambda \in \Lambda} (G+\lambda) = \R^d,$ and each element of $G+\Lambda$ contains $|G\cap X|$ points of $X$. Let
     $c \coloneqq \frac{|G \cap X|}{\leb(G)}>0,$ and apply Proposition \ref{5.5 extended} with $\nu_n =\mu_n,$ $\nu = c\leb$, $\mathcal{T}_n \coloneqq\{\varphi_n^{-1}(G+\lambda)\cap[0,1]^d:\lambda \in \Lambda\}$, and $\widetilde{\mathcal{T}_n}\coloneqq  \mathcal{T}_n\cap\{\varphi_n^{-1}(G +\lambda):\lambda \in \Lambda\}.$ Note that $\widetilde{\mathcal{T}_n}$ is the collection of all sets $\varphi_{n}^{-1}(G+\lambda)$ with $\lambda\in\Lambda$ which are contained in $[0,1]^{d}$. Let us check the four conditions of \eqref{prop conditions} in Proposition~\ref{5.5 extended}. The first condition follows from the fact that $\bigsqcup_{\lambda \in \Lambda} (G+\lambda) = \R^d$ and that for each $n \in \N$, $\varphi_n^{-1}$ is an bijective affine map, so $\mathcal{T}_n$ is a partition of $[0,1]^d.$
     The third and fourth conditions are immediate since for each $n \in \N$ and $T \in {\mathcal{T}_n},$ writing $T = \varphi_n^{-1}(G+\lambda)\cap[0,1]^d$ for some $\lambda \in \Lambda,$ we see that $$\diam(T) \leq \diam(\varphi_n^{-1}(G+\lambda)) = \frac{1}{r_n}\diam(G) \xrightarrow{n \to \infty} 0$$ and
if further we have that $T \in \widetilde{T}_n,$ then     
     $$|\mu_n(T)-c\leb(T)| = \bigg|\frac{1}{r_n^d}|(G+\lambda) \cap X| -c\leb\left(\varphi_n^{-1}(G+\lambda)\right)\bigg|=\bigg|\frac{1}{r_n^d}|G \cap X| -c\frac{\leb(G)}{r_n^d}\bigg| =0.$$
     We must verify that the second condition
     \begin{equation} \label{eq: max condition}
         \lim_{n \to \infty} \max\Bigg\{\mu_n\left([0,1]^d\setminus\bigcup_{T \in \widetilde{T}_n} T\right),c\leb\left([0,1]^d\setminus\bigcup_{T \in \widetilde{T}_n} T\right)\Bigg\} = 0
     \end{equation}
      holds. We claim for each $n \in \N$ that \begin{equation} \label{eq: Hn}
          H_n \coloneqq 
           \left[\frac{2\diam(G)}{r_n},1-\frac{2\diam(G)}{r_n}\right]^{d} \subseteq \bigcup_{T \in \widetilde{\mathcal{T}_n}}T.
      \end{equation} Indeed, for each $n \in \N$ and $\lambda \in \Lambda,$ we have that
     $$\text{Dist}(\R^d \setminus[0,1]^d,H_n) \geq \frac{2\diam(G)}{r_n}>\frac{\diam(G)}{r_n}=\diam\left(\varphi_n^{-1}\left(G+\lambda\right)\right),$$ so each element of $\mathcal{T}_n\setminus \widetilde{\mathcal{T}_n}$ is a subset of $[0,1]^d \setminus H_n$. Hence, we infer that \eqref{eq: Hn} holds. Now notice that $\leb(H_n) = \left(1-\frac{4\diam(G)}{r_n}\right)^d \xrightarrow{n \to \infty} 1,$ and $\varphi_n([0,1]^d \setminus H_n)$ is equal to a pairwise disjoint, finite union of closed cuboids of sidelengths at least $2\diam(G) \geq l,$ where $l$ is the packing radius of $X$. Hence, by Lemma \ref{claim: K}, there exists a constant $K=K(X)>0$ such that $$\mu_n([0,1]^d \setminus H_n) = \frac{1}{r_n^d}|\varphi_n([0,1]^d \setminus H_n) \cap X| \leq \frac{1}{r_n^d}K\leb (\varphi_n([0,1]^d \setminus H_n)) = K\leb([0,1]^d \setminus H_n)\xrightarrow{n \to \infty} 0.$$ Therefore, using \eqref{eq: Hn}, we have that $[0,1]^d \setminus \bigcup_{T \in \widetilde{T}_n}T \subseteq [0,1]^d \setminus H_n,$ so we have shown that \eqref{eq: max condition} holds. By Proposition \ref{5.5 extended}, we infer that $\mu_n \rightharpoonup c\leb.$ By the uniqueness of weak limits, we obtain that $\rho\leb = c \leb,$ i.e., $\rho = c$ almost everywhere.
        
\end{proof}
\begin{proof}[Proof of Theorem \ref{thm: pn rep}]
    The sequence $(p_n)_{n \in \N}$ is well-defined since $X$ is repetitive. 
    Since $X$ encodes $\rho,$ there exist sequences $(r_n)_{n \in \N} \subset \R_{>0}$ going to $\infty$, closed cubes $(Q_n)_{n \in \N} \subset \R^d$ of sidelength $r_n,$ bijective affine mappings $\varphi_n:[0,1]^d \to Q_n$, and measures $(\mu_n)_{n \in \N}$ defined by 
    \begin{equation*}
         \mu_n(A) = \frac{1}{r_n^d}|X \cap \varphi_n(A)|, \qquad A \subseteq [0,1]^d
    \end{equation*} satisfying that $\mu_n \rightharpoonup \rho\mathcal{L}.$
    
    Notice that $(p_n)_{n \in \N}$ is a non-decreasing sequence of natural numbers. Moreover, the definition of $(p_{n})_{n\in\N}$ implies that either $(p_n)_{n\in\N}$ is eventually constant, or $(p_n)_{n\in\N}$ is strictly increasing. If $(p_n)_{n \in \N}$ is eventually constant, then by Lemma~\ref{lemma: pn constant X ideal}, $X$ is an ideal crystal, and so by Lemma \ref{lemma: X ideal rho constant}, $\rho$ is constant almost everywhere.
 Therefore, we may assume that $(p_{n})_{n\in\N}$ is strictly increasing (and tends to $\infty$). In this case, notice that for each $n \in \N$
 \begin{equation} \label{eq: p0}
     2^{p_{n-1}}\geq 2^{p_0} \geq 4L.
 \end{equation}
 For the remainder of this proof, we will consider \emph{half-open cubes}, by which we mean sets of the form $$\prod_{i=1}^d(x_i,x_i+s], \qquad x_1,...,x_d \in \R, \;s>0.$$ The standard partition of a large half-open cube into smaller half-open cubes has the property that the union of all elements in the partition cover the large cube, while elements of the partition are pairwise disjoint.  
  
  Define the sequences of the \emph{lightest} and \emph{darkest shades} of $X$, $(a_n)_{n\in \N}$ and $(b_n)_{n\in \N}$ respectively, to be
    \begin{equation*}
       a_n \coloneqq \frac{1}{2^{dp_{n-1}}}\inf_{S \in \mathcal{S}_n}|X \cap S| \qquad \text{and} \qquad b_n \coloneqq \frac{1}{2^{dp_{n-1}}}\sup_{S \in \mathcal{S}_n}|X \cap S|,
    \end{equation*} where $\mathcal{S}_n$ is the collection of half-open cubes in $\R^d$ of sidelength $2^{p_{n-1}}.$ 
    The numbers $a_n$ and $b_n$ are well-defined since for each $n\in \N$ and each cube $S \in \mathcal{S}_n$, we have that $2^{p_{n-1}}\geq 4L \geq l$ by \eqref{eq: p0} and so $$0 \leq |X\cap S|\leq K\leb(S) = K2^{dp_{n-1}},$$ where $K$ is the constant from Lemma \ref{claim: K}. Therefore, for each $n\in \N$ we have that $a_n \geq 0$ and $b_n \leq K$.
    \begin{claim} \label{claim: an bn}
        $(a_n)_{n\in \N}$ is non-decreasing and $(b_n)_{n\in \N}$ is non-increasing.
    \end{claim}
    \begin{proof}
        Fix $n\in \N$. Then considering the standard partition of any half-open cube of sidelength $2^{p_n}$ into $2^{d(p_n-p_{n-1})}$ half-open cubes of sidelength $2^{p_{n-1}}$
        \begin{equation*}
            a_{n+1} = \frac{1}{2^{dp_{n}}}\inf_{S \in \mathcal{S}_{n+1}} |X \cap S| \geq \frac{1}{2^{dp_{n}}}2^{d(p_{n}-p_{n-1})}\inf_{S \in \mathcal{S}_n} |X \cap S|= a_{n}.
        \end{equation*}  A similar calculation holds for $b_n.$
    \end{proof}
    Claim~\ref{claim: an bn} together with the fact that $(a_{n})_{n\in\N}$ and $(b_{n})_{n\in\N}$ are bounded implies that $\lim_{n \to \infty} a_n$ and $\lim_{n \to \infty} b_n$ exist. In addition, it is clear by definition that $a_n \leq b_n$ for each $n.$ 
    Now fix $\epsilon >0$ and let $N = N(X,\epsilon) \in \N$ to be determined later. Fix $n \geq N.$ By the definition of $b_{n}$, there exists a half-open cube $E=E_{n,\epsilon} \in \mathcal{S}_n$ such that 
    \begin{equation} \label{eqn: Qeps}
        \frac{1}{2^{dp_{n-1}}}|X \cap E| \geq b_n - \frac{\epsilon}{2}.
    \end{equation}
    Since $X$ has covering radius $L,$ we have that every ball in $\R^d$ of radius $L$ contains at least one point of $X.$ Thus, the cube of sidelength $2L$ centred at the centre of $E$ contains a point $x\in X,$ and is contained in $E$ by \eqref{eq: p0}.
    Consider the open cube $P = P_{n,x}$ of sidelength $2^{p_{n-1}}$ centred at $x.$ To estimate the difference between the number of points in $X \cap E$ and $X \cap P$, observe that $P\cap E$ contains an open cube $\widetilde{E}$ of sidelength $2^{p_{n-1}}-L$ and with a common vertex with $E$. Thus $E\setminus P$ is covered by $E\setminus \widetilde{E}$.
    Notice too that $E \setminus \widetilde{E}$ is the finite union of cuboids with sidelengths all at least $L.$ Note $L \geq l$, so with $K >0$ as in Lemma \ref{claim: K} 
    \begin{equation} \label{eq:shifted square}
    \begin{split}
        \frac{1}{2^{dp_{n-1}}}|X \cap E| - \frac{1}{2^{dp_{n-1}}}|X \cap P| &\leq \frac{1}{2^{dp_{n-1}}}|X \cap (E\setminus P)| \leq \frac{1}{2^{dp_{n-1}}}|X \cap (E\setminus \widetilde{E})| \leq \frac{1}{2^{dp_{n-1}}}K\leb(E\setminus \widetilde{E})\\&\leq\frac{K}{2^{dp_{n-1}}}((2^{p_{n-1}})^d-(2^{p_{n-1}}-L)^d) \leq K\left(1-\left(1-\frac{L}{2^{p_{N-1}}}\right)^d\right)<\frac{\epsilon}{2},
        \end{split}
    \end{equation} 
    where the last inequality is the first condition we impose on $N=N(X,\varepsilon)$.
    Therefore, by \eqref{eqn: Qeps} and \eqref{eq:shifted square}
    \begin{equation*}
        \frac{1}{2^{dp_{n-1}}}|X \cap P| = \frac{1}{2^{dp_{n-1}}}|X \cap E| + \left(\frac{1}{2^{dp_{n-1}}}|X \cap P| - \frac{1}{2^{dp_{n-1}}}|X \cap E|\right) \geq b_n-\frac{\epsilon}{2} - \frac{\epsilon}{2}
 \geq b_n - \epsilon.
    \end{equation*} 
    Now we will show $\lim_{n\to \infty}a_n = \lim_{n \to \infty} b_n$. Suppose for a contradiction that this assertion is false. Fix $n \geq N$ and consider an arbitrary half-open cube $S \in \mathcal{S}_{n+1}$ of sidelength $2^{p_n}$. By the definition of $p_n,$ $S$ contains an $X-$translate, which we will denote by $\mathcal{R}$, of the cubic patch $\mathcal{P}^{(Q)}_{x,2^{p_{n-1}}} = X \cap P.$ The cubic patch $\mathcal{R}$ is contained in a half-open cube of sidelength $2^{p_{n-1}}$, which we call $\mathcal{U}_1$. Consider the standard partition of $S$ into $2^{d(p_{n}-p_{n-1})}$ half-open cubes of sidelength $2^{p_{n-1}}$. The cubic patch $\mathcal{R}$ is contained within the union $\mc{U}$ of a collection of $2^d$ half-open cubes in this partition. We further consider the standard partition of $\mathcal{U}$ into $2^{d(p_{n-1}-p_{m}+1)}$ half-open cubes of sidelength $2^{p_m},$ where $m \coloneqq \lfloor\frac{n-1}{2}\rfloor.$ There are at most $(2^{p_{n-1}-p_{m}}+1)^d$ elements of this partition that intersect $\mathcal{U}_1,$ so there are at least $2^{d(p_{n-1}-p_m+1)}-(2^{p_{n-1}-p_{m}}+1)^d$ half-open cubes of this partition that do not intersect $\mathcal{U}_1$. Refer to the union of these half-open cubes not intersecting $\mathcal{U}_1$ as $\mathcal{U}_2$. Then 
    \begin{equation*} \label{eq: patch an+1}
    \begin{split}
        \frac{1}{2^{dp_n}} &|X \cap S|= \frac{1}{2^{dp_n}}(|X \cap (S \setminus \mathcal{U})|+|X \cap \mathcal{U}|) \geq \frac{1}{2^{dp_n}}(|X \cap (S \setminus \mathcal{U})|+|X \cap \mathcal{U}_1|+|X \cap \mathcal{U}_2|)
        \\&\geq \frac{1}{2^{dp_{n}}}((2^{d(p_{n}-p_{n-1})}-2^d)2^{dp_{n-1}}a_{n} + 2^{dp_{n-1}}(b_n - \epsilon) + (2^{d(p_{n-1}-p_m+1)}-(2^{p_{n-1}-p_{m}}+1)^d)a_{m+1}2^{dp_m}) \\
        &= a_n + 2^{d(p_{n-1}-p_{n})}(b_n-a_n -\epsilon + (2^d-1)(a_{m+1}-a_n) + (1 - (1+2^{p_{m}-p_{n-1}})^d)a_{m+1})\\
        &\geq a_n + 2^{d(p_{n-1}-p_{n})}(b_n-a_n -\epsilon-(2^{d}-1)(\lim_{i\to\infty}a_{i}-a_{\floor{\frac{N-1}{2}}})-K((1+2^{-\frac{N-1}{2}})^{d}-1))\\
        &\geq a_n + 2^{d(p_{n-1}-p_n)}(b_n-a_n-2\epsilon),
    \end{split}
    \end{equation*} 
    In the penultimate inequality, we use that $(p_{i})_{i\in\N}$ is strictly increasing, which implies $p_{n-1}-p_{m}\geq n-1-m\geq \frac{N-1}{2}$. We also use that $a_{j}\leq a_{j+1}\leq \lim_{i\to\infty}a_{i}\leq K$ for all $j\in\N$. The last inequality in the sequence above is equivalent to $$(2^{d}-1)\br{\lim_{i\to\infty}a_{i}-a_{\floor{\frac{N-1}{2}}}}-K((1+2^{-(N-1)/2})^{d}-1)\leq\epsilon,$$ which is the second and final condition we impose on $N = N(X,\epsilon).$
    Since $S$ is arbitrary, we have proved that $$a_{n+1} = \inf_{S \in \mathcal{S}_{n+1}} \frac{1}{2^{dp_n}}|X \cap S| \geq a_n + 2^{d(p_{n-1}-p_n)}(b_n-a_n-2\epsilon) \qquad \text{for each $n \geq N$}.$$
    To summarise, for any $\epsilon>0$ we have shown that there exists $N(X,\epsilon)\in\N$ such that $$a_{n+1}\geq a_{n}+2^{d(p_{n-1}-p_{n})}(b_{n}-a_{n}-2\epsilon)\qquad \text{for each $n\geq N(X,\epsilon)$.}$$
    Now take $\epsilon \coloneqq \frac{1}{3}\lim_{n \to \infty}(b_n-a_n)>0.$ Since $b_n-a_n$ is non-increasing
$$b_n-a_n-2\epsilon \geq \lim_{n \to \infty}(b_n-a_n) - 2\epsilon = 3\epsilon - 2\epsilon = \epsilon, \qquad \text{for each $n\geq N(X,\epsilon)$.}$$ Therefore, for each $n \geq N=N(X,\epsilon)$
    $$a_{n+1} \geq \epsilon\sum_{j=N}^n2^{d(p_{j-1}-p_j)},$$ and so 
    $\lim_{n \to \infty} a_{n} \geq \epsilon\sum_{j = N}^{\infty}2^{d(p_{j-1}-p_{j})} = \infty,$ which is a contradiction. Therefore, $\lim_{n \to \infty}a_n = \lim_{n \to \infty} b_n.$ 

    Finally, we must show that $\rho$ is constant almost everywhere. We will show that $\mu_n \rightharpoonup a\mathcal{L}|_{[0,1]^d},$ where $a \coloneqq \lim_{n \to \infty}a_n$, and so by the uniqueness of weak limits we will have that $\rho\mathcal{L} = a\mathcal{L}|_{[0,1]^d},$ which means $\rho = a$ almost everywhere. To this end, we will use the Proposition \ref{5.5 extended} with $\mathcal{K} = [0,1]^d$, $\nu = a\mathcal{L}|_{[0,1]^d}$, and $\nu_n = \mu_n$ for each $n \in \N.$ Choose $J \in \N$ so that $r_n \geq 1$ for each $n \geq J$. For each $n \geq J,$ let $m \in \N$ be such that $2^{p_{m-1}} \leq \sqrt{r_n} < 2^{p_m}$. Then take the finite collection $\mathcal{T}_n$ to be 
    \begin{equation*}
        \mathcal{T}_n \coloneqq \bigg\{[0,1]^d \cap\prod_{i=1}^d [x_i,x_i+s_n]: (x_1,...,x_d) \in s_n\Z^d \bigg\}, \qquad\text{with}
        \; s_n \coloneqq \frac{2^{p_{m-1}}}{r_n}.
    \end{equation*} Since $s_n \leq 1$ for each $n \geq J$, note for each $n \geq J$ that
    $$|\mathcal{T}_n| =|s_n\Z^d \cap [0,1]^d| = \left(\left\lfloor\frac{1}{s_n}\right\rfloor+1\right)^d\leq \left(\frac{2}{s_n}\right)^d.$$
    $\mathcal{T}_n$ $\nu-$covers $[0,1]^d$, and since $s_n \leq \frac{1}{\sqrt{r_n}} \xrightarrow{n\rightarrow\infty}0$, the diameter of each element of $\mathcal{T}_n$ goes to $0$ as $n \to \infty.$ To show the last condition of (\ref{prop conditions}), we will choose $\widetilde{\mathcal{T}_n}$ to be the finite collection of cubes in $\mathcal{T}_n$ of sidelength exactly $s_n.$ It is easy to check that $\widetilde{\mathcal{T}_n}$ obeys the conditions of Proposition \ref{5.5 extended} since $s_n \xrightarrow{n\rightarrow\infty}0$. For any cube $T \in \widetilde{\mathcal{T}_n},$ notice that 
    \begin{equation*}
        \begin{split}
            \mu_n(T) &= \frac{1}{r_n^d}|X \cap \varphi_n(T)| = \frac{s_n^d}{(r_ns_n)^d}|X \cap \varphi_n(T)|
            \leq \mathcal{L}(T)\frac{1}{2^{dp_{m-1}}} \sup_{S\in \mathcal{S}_m} |X \cap S| = \mathcal{L}(T)b_m \\&= a\mathcal{L}(T) + (b_m-a)\mathcal{L}(T) = a\mathcal{L}(T) + (b_m-a)s_n^d \leq a\mathcal{L}(T) + 2^d\frac{b_m-a}{|\mathcal{T}_n|}.
        \end{split}
    \end{equation*} A similar calculation gives that $\mu_n(T) \geq a\mathcal{L}(T) - 2^d\frac{a-a_m}{|\mathcal{T}_n|}.$ But notice that $m \xrightarrow{n\rightarrow\infty} \infty$, so both $b_m-a$ and $a-a_m$ are in $o(1),$ and so $|\mu_n(T)-a\mathcal{L}(T)| \in o(\frac{1}{|\mathcal{T}_n|}).$ Therefore, $\rho = a$ almost everywhere.
\end{proof}
\rlogr*
    \begin{proof}
        The theorem will be proved as follows. Define the sequence $(p_n)_{n \in \N \cup \{0\}}$ as in \eqref{eq: defining sequence pn}. We will show $\sum_{n = 1}^\infty 2^{-d(p_n-p_{n-1})} = \infty.$ Then by Theorem \ref{thm: pn rep}, $\rho$ is constant almost everywhere.
\begin{claim} \label{claim: Rpn} For all $n \in \N$
    $$R\left(\frac{\sqrt{d}}{2}2^{p_{n-1}}\right)+\sqrt{d}\;2^{p_{n-1}} \geq \frac{1}{4}2^{p_n}.$$
\end{claim}
\begin{proof}
    Fix $n \in \N$. By the definition of $p_n,$ there exists $x \in X$, a $2^{p_{n-1}}$-cubic patch $\mathcal{P}^{(Q)}_x$ centred at $x$, and an open cube $Q$ of sidelength $2^{p_{n}-1}$ in $\R^d$ such that $Q$ does not contain any $X-$translate of $\mathcal{P}^{(Q)}_x$. Consider the inscribed open ball $B$, of radius $2^{p_n-2},$ in the cube $Q.$ Notice $B$ does not contain an $X-$translate of the $2^{p_{n-1}}-$cubic patch $\mathcal{P}^{(Q)}_x$. Since open cubes of sidelength $2^{p_{n-1}}$ are inscribed in an open ball of sidelength $\frac{\sqrt{d}}{2}2^{p_{n-1}},$ we have that $B \cap X$ does not contain an $X-$translate of the $\frac{\sqrt{d}}{2}2^{p_{n-1}}-$patch $\mathcal{P}_x$ which is centred at $x.$ By the definition of the repetitivity function $R$, every open ball of radius $R(\frac{\sqrt{d}}{2}2^{p_{n-1}})+{\sqrt{d}}\;2^{p_{n-1}}$ in $\R^d$ must contain an $X-$translate of every $\frac{\sqrt{d}}{2}2^{p_{n-1}}-$patch, so  $$R\left(\frac{\sqrt{d}}{2}2^{p_{n-1}}\right)+\sqrt{d}\;2^{p_{n-1}} \geq \text{radius}(B)=2^{p_n - 2} = \frac{1}{4}2^{p_n}.$$
\end{proof}
Using Claim \ref{claim: Rpn} and that $R(r) \in O(r(\log r)^{\frac{1}{d}})$, there exists a constant $C_1>0$ such that for all $n \in \N$
\begin{equation*} \label{eqn: rep pn}
    \frac{1}{4}2^{p_n-p_{n-1}}\leq \frac{R(\frac{\sqrt{d}}{2}2^{p_{n-1}})}{2^{p_{n-1}}}+\sqrt{d} \leq C_1\frac{\sqrt{d}}{2} \left(\log \left(\frac{\sqrt{d}}{2}2^{p_{n-1}}\right)\right)^{\frac{1}{d}}+\sqrt{d}.
\end{equation*}
Therefore, there exists some $N \in \N$ such that, setting $C_2(d)\coloneqq 4C_1{\sqrt{d}}(2\log2)^\frac{1}{d}>0,$ for all $n\geq N$ 
\begin{equation} \label{eq: order bound}
    2^{p_{n}-p_{n-1}}\leq {8C_1\frac{\sqrt{d}}{2}}(\log2^{2p_{n-1}})^\frac{1}{d} = 4C_1\sqrt{d}(2\log2)^{\frac{1}{d}}{p_{n-1}}^\frac{1}{d} = C_2{p_{n-1}}^\frac{1}{d}.
\end{equation}
\begin{claim} \label{claim: 2^pn is O(n^n)}
    Let $N \in \N$ be as in \eqref{eq: order bound}. Then there exists a constant $C\geq 1$ so that for each $n \geq N$  $$2^{p_{n-1}} \leq (Cn)^n.$$
\end{claim}
\begin{proof}
    Take $C \coloneqq \max\{2^{p_{N-1}}, {C_2}^2\},$ where $C_2$ is the constant in \eqref{eq: order bound}. The claim is true for $n=N$ since $C\geq 2^{p_{N-1}}$. Assume for a fixed $n \geq N$ that $2^{p_{n-1}} \leq (Cn)^n$. Then using \eqref{eq: order bound} along with the facts that $d\geq 2$ and $\log_2 x \leq x$ for every $x \geq 1$
    \begin{equation*}
    \begin{split}
        2^{p_n}&=2^{p_{n}-p_{n-1}}2^{p_{n-1}} \leq C_2{p_{n-1}}^\frac{1}{d}(Cn)^n \leq C^\frac{1}{2}(\log_2((Cn)^n))^\frac{1}{d}(Cn)^n \\
        &= C^\frac{1}{2}(n\log_2(Cn))^\frac{1}{d}(Cn)^n \leq C^\frac{1}{2}(Cn^2)^\frac{1}{d}(Cn)^n \leq (Cn)^{n+1}\leq (C(n+1))^{n+1}.
        \end{split}
    \end{equation*}
\end{proof}
Combining the bound in \eqref{eq: order bound} with Claim \ref{claim: 2^pn is O(n^n)}, we get for each $n \geq N$ that 
$$2^{-d(p_n-p_{n-1})} \geq C_2^{-d}{p_{n-1}}^{-1} \geq C_2^{-d}(\log_2((Cn)^n))^{-1} = C_2^{-d}(n\log_2(Cn))^{-1}.$$ Therefore,
\begin{equation*}
    \sum_{n=1}^\infty 2^{-d(p_n-p_{n-1})} \geq \sum_{n=N}^\infty 2^{-d(p_n-p_{n-1})} \geq C_2^{-d}\sum_{n=N}^\infty\frac{1}{n\log_2(Cn)} = \infty.
\end{equation*} 
    \end{proof}
Notice that the class of functions in which $R(r)$ lies in Theorem \ref{thm: repetitivity implies constant rho} is smaller than the class of functions in which $R(r)$ lies in Theorem \ref{delone}. This is because of the discrete nature of how we calculate the repetitivity function based on estimating the radius $r$ of a given patch in our Delone set: the repetitivity function in our construction is of the form $\sqrt{d} \; 2^{p_{n}}$ whenever $2^{p_{n-2}}<2r \leq 2^{p_{n-1}},$ with the upper bound used in Theorem \ref{thm: repetitivity implies constant rho} and the lower bound used in Theorem \ref{delone}. It may be the case that there is a non-rectifiable $O(r(\log{r})^{\frac{1}{d}})-$repetitive Delone set in $\R^d$, but our method is only able to determine an exponent of roughly $\frac{2}{d}$, so a new idea may be needed to bridge the small remaining gap. 

\paragraph{Acknowledgements} The authors would like to thank Vojt\v ech Kalu\v za for helpful discussions. Moreover, they would like to thank the anonymous referee for careful reading and helpful corrections that improved the paper. AB thanks the University of Birmingham for financial support.

\begin{flushleft}
    Michael Dymond,\\
    School of Mathematics, University of Birmingham, Birmingham, B15 2TT, United Kingdom.\\
    \href{mailto:m.dymond@bham.ac.uk}{m.dymond@bham.ac.uk} 
\end{flushleft}

\begin{flushleft}
    Ashwin Bhat,\\
    School of Mathematics, University of Birmingham, Birmingham, B15 2TT, United Kingdom.\\
    \href{mailto:axb1805@student.bham.ac.uk}{axb1805@student.bham.ac.uk} 
\end{flushleft}

\end{document}